\documentclass{amsart}
\usepackage{amsmath}
\usepackage{amsfonts}
\usepackage{amssymb}
\usepackage{amsthm}
\usepackage{epsfig,epstopdf}
\usepackage[dvipsnames,usenames]{color}
\usepackage{hyperref}
\usepackage{xifthen}
\usepackage{graphicx}
\usepackage{tikz}
\usepackage[colorinlistoftodos]{todonotes}
\tikzstyle{vertex}=[circle, draw, inner sep=0pt, minimum size=6pt]

\usepackage[
top    = 1in,
bottom = 1in,
left   = 1in,
right  = 1in]{geometry}
\usepackage{graphicx}

\tikzset{
  treenode/.style = {align=center, inner sep=0pt, text centered,
    font=\sffamily},
  arn_n/.style = {treenode, circle, black, text width=1em},
  arn_r/.style = {treenode, circle, blue, draw=blue, 
    text width=1em, very thick}}

\theoremstyle{plain}
\newtheorem{theorem}{Theorem}[section]
\newtheorem{lemma}[theorem]{Lemma}
\newtheorem{corollary}[theorem]{Corollary}
\newtheorem{proposition}[theorem]{Proposition}

\theoremstyle{definition}

\newtheorem{defi}[theorem]{Definition}

\newcommand{\C}{\mathbb{C}}
\newcommand{\D}{\mathbb{D}}
\newcommand{\N}{\mathbb{N}}

\newcommand{\R}{\mathbb{R}}

\newcommand{\ds}{\displaystyle}
\newcommand{\Lip}{\mathcal{L}}
\renewcommand{\mod}[1]{\left|#1\right|}

\newcommand{\calB}{\mathcal{B}}
\newcommand{\calF}{\mathcal{F}}

\newcommand{\calM}{\mathcal{M}}

\newcommand{\dist}{\operatorname{d}}

\newcommand{\parent}[1]{\operatorname{par}({#1})}
\newcommand{\nparent}[2]{\operatorname{par}^{{#2}}({#1})}
\newcommand{\Chi}[1]{{\operatorname{Chi}}({#1})}
\newcommand{\nChi}[2]{{\operatorname{Chi}}^{#2}({#1})}

\newcommand{\Lp}{\mathbf L^p}

\renewcommand\root{\operatorname{o}}

\allowdisplaybreaks

\begin{document}
\date{\today}

\title[Shifts on $\Lip$] 
 {The Forward and Backward Shift on the Lipschitz Space of a Tree}

\author[R.~A. Mart\'inez-Avenda\~no]{Rub\'en A. Mart\'inez-Avenda\~no}
\address{Departamento Acad\'emico de Matem\'aticas \\ Instituto
  Tecnol\'ogico Aut\'onomo de M\'exico \\ Mexico City, Mexico}
\email{rubeno71@gmail.com}
\author[E. Rivera-Guasco]{Emmanuel Rivera-Guasco}
\address{Centro de Investigaci\'on en Matem\'aticas \\ Guanajuato, Mexico}
\email{emmanuel.rivera@cimat.mx}

\thanks{We would like to thank the
  reviewers for their suggestions, which greatly improved this
  paper. The first author's research is partially supported by the
  Asociaci\'on Mexicana de Cultura A.C}

\begin{abstract} 
We initiate the study of the forward and backward shifts on the
Lipschitz space of an undirected tree, $\Lip$, and on the little Lipschitz space of
an undirected tree, $\Lip_0$. We determine that the forward shift is bounded both
on $\Lip$ and on $\Lip_0$ and, when the tree is leafless, it is an
isometry; we also calculate its spectrum. For the backward shift, we
determine when it is bounded on $\Lip$ and on $\Lip_0$, we find the
norm when the tree is homogeneous, we calculate the spectrum for the
case when the tree is homogeneous, and we determine, for a general
tree, when it is hypercyclic.
\end{abstract}

\maketitle

\section{Introduction}\label{sec_intro}

In \cite{CoEa1}, Colonna and Easley introduced the Lipschitz space of
a tree, $\Lip$. This is the Banach space of complex-valued
functions on a rooted, countably infinite, locally finite and
undirected tree (from now on simply referred to as a tree) which are
Lipschitz functions, when the tree is endowed with the edge-counting
metric. This space may be considered as the discrete analogue of the
classical Bloch space: the space of functions $f:\D \to \C$ which are
Lipschitz when the unit disk $\D$ is given the hyperbolic or Bergman
metric (see, e.g., \cite{zhu}) and the set of complex numbers $\C$ is
given the usual Euclidean metric.

As it turns out, the Lipschitz space of the tree is, roughly speaking, the
space of funtions on the tree whose ``derivative'' remains bounded on
the tree. Therefore, there is also the little Lipschitz space,
$\Lip_0$, defined as the space of functions on the tree whose
``derivative'' tends to zero when far away from the root of the tree
(i.e.,  on the ``boundary'' of the tree).

The motivation for investigating spaces of functions on trees comes
mainly from harmonic analysis. Early studies of harmonic functions on
regular trees were done by Cartier in \cite{cartier1, cartier2}. Also, Cohen
and Colonna studied the Bloch space of harmonic functions on a
regular tree in \cite{CoCo1}, characterizing several properties of functions on this
space. Later, in \cite{CoCo2} Cohen and Colonna showed how to embed
certain homogeneous trees in the hyperbolic disk in a ``nice way'':
for example, in such a way that bounded harmonic functions on the disk
correspond to harmonic functions on the tree.

Several operators on the Lipschitz space of a tree have been
studied. For instance, in \cite{CoEa1}, Colonna and Easley characterize
boundedness of multiplication operators on $\Lip$ and on $\Lip_0$, as
well as establishing other operator-theoretical properties of such
operators. In \cite{ACE4}, Allen, Colonna and Easley study properties
of the composition operators on the Lipschitz space of a tree. There
have also been several studies of multiplication and other operators
defined on $\Lip$ and on other Banach spaces on trees \cite{ACE1,
  ACE2, ACE3, CoEa2, CoMA1, CoMA2}.

The shift operators (both the forward and backward shifts) on $\ell^p$
have been studied for a long time. There are several reasons why
researchers have been interested in shift operators: one of them is
that they provide a wealth of examples and counterexamples in operator
theory (see, e.g., \cite{shields}). In \cite{JJS},
Jab{\l}o\'nski, Jung and Stochel initiated the study of shifts on
directed trees. In their paper, they investigate several operator
theoretic properties of weighted (forward) shifts on the $L^2$ space of an
infinite directed tree. Later, in \cite{MA}, the first author defined
(inspired by a result in \cite{JJS}) the backward shift operator on a
weighted $L^p$ space of a directed tree and characterized its hypercyclicity.

The study of hypercyclic operators goes as far back as the papers of
Birkhoff~\cite{birkhoff} and MacLane~\cite{maclane}, but the first
example of a hypercyclic operator on a Banach space was given by
Rolewicz~\cite{rolewicz}: it is a multiple of the backward shift on
$\ell^p$. For the basic definitions and the history of hypercycicity,
we recommend the texts \cite{GEP} and \cite{BaMa}.

The purpose of this paper is to introduce the study of the forward and
backward shift operator on the Lipschitz space $\Lip$ and on the little
Lipschitz space $\Lip_0$. The paper is organized as follows. After
giving the basic definitions and notations we will use throughout
this paper in Section \ref{sec_prelim}, we define the forward and
backward shifts in Section \ref{sec_forward_shift}. We observe that
the forward shift is always an isometry, when the tree is leafless,
and find its spectrum when it acts on $\Lip$ and on $\Lip_0$. Also, we
establish that the backward shift is the adjoint operator of the
forward shift. In Section \ref{sec_backward_shift}, we give a
sufficient and necessary condition to ensure that the backward shift
is bounded: it turns out the backward shift is bounded exactly when
the tree is homogeneous by sectors. We summarize this characterization
in Theorem~\ref{th:main}, and we also find some lower estimates for
the norm. Later, in Section \ref{sec_norm},
we find the value of the norm of the backward shift and the value of
the norm for powers of the backward shift, in the case of
homogeneous trees. In Section \ref{sec_spectrum}, we obtain the
spectrum for the backward shift in the case where the tree is
homogeneous, both for $\Lip$ (Theorem~\ref{th:sigmaB_L}) and for
$\Lip_0$ (Theorems~\ref{th:sigmaB_L0} and~\ref{th:sigmap_L0}). Lastly, in Section
\ref{sec_hyper}, we establish that the forward shift can never be
hypercyclic, but the backward shift is hypercyclic exactly when the tree has no free ends
(Theorem~\ref{th:main_hyper}): this result is analogous to the one found in
\cite{MA}.

\section{Preliminaries}\label{sec_prelim}

As is customary, $\N$, $\N_0$, $\R$, $\C$ and $\D$ will denote the set of
natural numbers, the set of nonnegative integers, the set of real
numbers, the set of complex numbers, and the open unit disk in $\C$
centered at the origin, respectively.

Recall that a {\em graph} $G=(V,E)$ consists of a nonempty set of {\em vertices} $V$
and a set of {\em edges} $E \subseteq \left\{ \{ u, v \} \, : \, u, v \in V, u \neq v
\right\}$. In this paper, the set of vertices $V$ will always be countably
infinite. If $\{u,v \} \in E$, we say that $u$ and $v$ are {\em adjacent} and we
denote this by $u \sim v$. For each $u \in V$, the {\em degree} of
$u$, denoted by $\deg(u)$, is the number of vertices adjacent to
$u$. In this paper, all of our graphs will be {\em locally finite}; i.e.,
$\deg(u) < \infty$ for every $u \in V$.

A {\em path of lenght $n$} joining two vertices $u$ and $v$ is a
finite sequence of $n+1$ distinct vertices $u=u_0 \sim u_1 \sim u_2
\sim \cdots \sim u_n=v$. A graph is a {\em tree} if for each pair of
vertices there is one and only one path between them. In this paper, for a tree $T$ we will
denote its set of vertices also by the letter $T$, which should cause no confusion.

Every tree $T$ we consider here has a distinguished vertex, which we call
the {\em root} of $T$ and denote by $\root$. For a tree $T$, we denote
by $\dist(u,v)$ the length of the unique path between the vertices $u,
v \in T$. For $v \in T$ we use the notation $|v|:=\dist(\root,v)$. We also
denote by $T^*$ the set of all vertices minus the root; i.e., $T^*:=\{
v \in T \, : \, |v|\geq 1 \}$. Ocasionally, we will denote by $T^{**}$
the set $\{ v \in T \, : \, |v|\geq 2 \}$.

For each $v \in T^*$, we define the {\em parent} of $v$, denoted by
$\parent{v}$, as the unique vertex $w$ in the path from $\root$ to $v$
with $|w|=|v|-1$. Observe that every vertex in $T$ has a parent,
except for the root $\root$. Inductively, for $n \in \N$, we define the {\em
  $n$-parent} of $v$, denoted by $\nparent{v}{n}$, as follows:
$\nparent{v}{1}:=\parent{v}$ if $v \neq \root$, and for $n\geq 2$, we set
$\nparent{v}{n}:=\parent{\nparent{v}{n-1}}$, if $v$ has a
$(n-1)$-parent and $\nparent{v}{n-1} \neq \root$. The set of all
vertices that have $n$-parents is denoted by $T^n$.

{\bf Note:} {\em We should point out that with the above definitions we
are giving the tree a directed structure, in the manner of the
definitions in \cite{JJS}: each edge can be thought of
as a directed edge going to a vertex $v \in T^*$ from its parent
$\parent{v}$. Also, we have chosen a fixed vertex and we called it a
root, which has no parent and thus coincides with the definition of a
root in \cite{JJS} (and hence it is unique). We choose not to follow
this point of view in this work, since the spaces we study ahead, were
originally defined on undirected trees.}

If $w$ is the parent of $v$, we say that $v$ is a {\em child} of
$w$ and we denote the set of all children of $w$ by $\Chi{w}$. If $w$ is
the $n$-parent of $v$ we say that $v$ is an {\em $n$-child} of $w$ and we
denote the set of all $n$-children of $w$ by $\nChi{w}{n}$. For a
vertex $v$, we denote by $\gamma(v)$ the number of children it has;
i.e., $\gamma(v)$ is the cardinality of $\Chi{v}$. Also, $\gamma(v,n)$
is the number of $n$-children of $v$; i.e., $\gamma(v,n)$ is the
cardinality of $\nChi{v}{n}$. We will say a tree is {\em homogeneous
  of order $\gamma$} if $\gamma(v)=\gamma$ for all $v\in T$ (this differs a bit from
the use of the term in the literature). 

If a vertex $v \in T$ satisfies that $\gamma(v)=0$ (i.e., $v$ has no
children) we will say that $v$ is a {\em leaf} of $T$. A tree with no
leaves will be called {\em leafless}. Observe that in a leafless tree,
every vertex is the parent of some other vertex.

Lastly, for every $v \in T$, we denote by $S_v$ the {\em sector}
determined by $v$, which consists of $v$ and all its $n$-children;
i.e., $S_v:= \bigcup_{n=0}^\infty \nChi{v}{n}$, where we will agree
that $\nChi{v}{0}=\{v\}$ and $\nChi{v}{1}=\Chi{v}$. Sometimes we will
refer to a sector as a {\em subtree}.

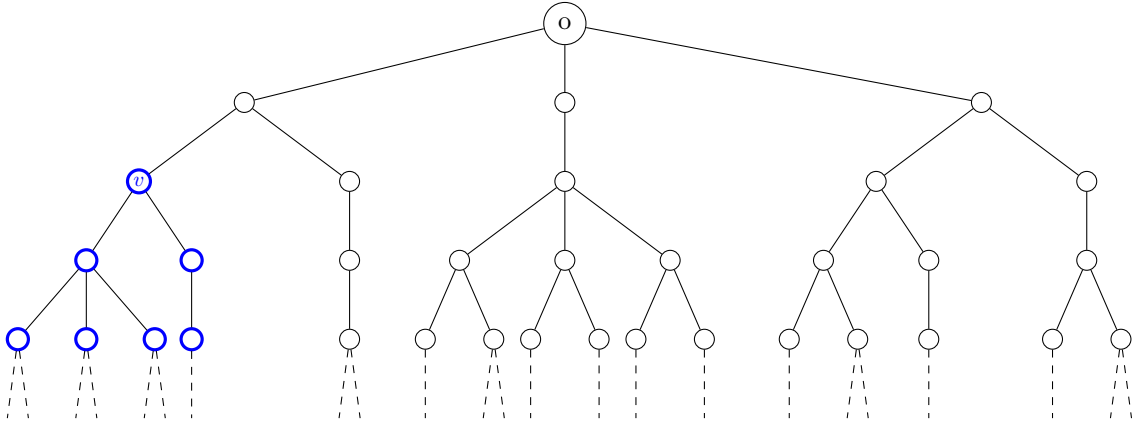
\begin{figure}[!htbp]
\centering
\begin{tikzpicture}[level 1/.style={sibling distance=7cm}, level
  2/.style={sibling distance=4cm}, level 3/.style={sibling
    distance=2cm}, level 4/.style={sibling distance=1.3cm}, level
  5/.style={sibling distance=0.4cm}, emph/.style={edge from parent/.style={black, dashed,draw}},
  scale=0.7]
  \node[draw,circle, scale=1] {$\root$} [grow'=down] 
child{node[draw,circle, scale=.8,right=0.5cm]  {} 
    child{node[draw,circle, scale=.8] {}
      child{node[draw,circle, scale=.8] {}
        child{node[draw,circle, scale=.8] {}
          child[emph]{}
          child[emph]{}}
        child{node[draw,circle, scale=.8] {}
                child[emph]{}}
              }
            }
        child{node[draw,circle, scale=.8] {}
           child{node[  draw,circle, scale=.8] {}
             child{node[  draw,circle, scale=.8] {}
               child[emph]{}}                 
                 }
             child{node[  draw,circle,
               scale=.8] {}
               child{node[  draw,circle, scale=.8] {}
                 child[emph]{}
                 child[emph]{}}
               child{node[  draw,circle, scale=.8] {}
                 child[emph]{}}                 
               }
            }}
child {node[draw,circle, scale=.8]  {} 
    child {node[draw,circle, scale=.8]{}
           child{node[  draw,circle, scale=.8] {}
               child{node[  draw,circle, scale=.8] {}
                 child[emph]{}}
               child{node[  draw,circle, scale=.8] {}
                 child[emph]{}}                 
                 }
              child{node[  draw,circle, scale=.8] {}
               child{node[  draw,circle, scale=.8] {}
                 child[emph]{}}
               child{node[  draw,circle, scale=.8] {}
                 child[emph]{}}                  
               }
             child{node[  draw,circle, scale=.8] {}
               child{node[  draw,circle, scale=.8] {}
                 child[emph]{}
                 child[emph]{}}
               child{node[  draw,circle, scale=.8] {}
                 child[emph]{}}                  
               }
            }
       }
child {node[draw,circle, scale=.8,right=0.5cm]  {}
  child{node[draw,circle, scale=.8] {}
    child{node[draw,circle, scale=.8] {} 
              child{node[draw,circle, scale=.8]{}
                child[emph]{}
                child[emph]{}
              }
            }
          }
    child{node[arn_r,draw,circle, scale=.8] {$v$} 
             child{node[arn_r, draw,circle, scale=.8] {}
             child{node[arn_r, draw,circle, scale=.8] {}
               child[emph]{}}                 
                 }
             child{node[arn_r, draw,circle, scale=.8] {}
               child{node[arn_r, draw,circle, scale=.8] {}
                 child[emph]{}
                 child[emph]{}}
               child{node[arn_r, draw,circle, scale=.8] {}
                 child[emph]{}
                 child[emph]{}}
               child{node[arn_r, draw,circle, scale=.8] {}
                 child[emph]{}
                 child[emph]{}}                 
            }
       }
     };
 \end{tikzpicture}
\caption{Sector at $v$: vertices in blue (and below).}
\end{figure}

Let $T$ be a countably infinite, locally finite tree. We denote by $\calF$ the set of all functions $f: T
\to \C$. In \cite{CoEa1}, Colonna and Easley define the Lipschitz
space of a tree as follows.

\begin{defi}
  The {\em Lipschitz space of $T$} is the set of all
  complex-valued Lipschitz functions on $T$; i.e., $f \in \calF$ is Lipschitz if
\[
\sup_{u\neq v} \frac{\mod{f(u)-f(v)}}{\dist(u,v)} < \infty.
\]
\end{defi}

In \cite{CoEa1}, Colonna and Easley show that
\[
\sup_{u\neq v} \frac{\mod{f(u)-f(v)}}{\dist(u,v)} = \sup_{v\in T^*} \mod{f(v)-f(\parent{v})}
\]
and hence the Lipschitz space consists of all functions for which
\[
\sup_{v\in T^*} \mod{f(v)-f(\parent{v})} < \infty.
\]
For $f \in \calF$, we define $f'$ as the function
\[
  f'(v)=\begin{cases}
    f(\root),& \text{ if } v = \root\\
    f(v)-f(\parent{v}),& \text{ if } v \neq \root.
  \end{cases}
\]
Thus $f$ is in the Lipschitz space if $f'$ is bounded. We denote
by $\Lip$ the set of all such functions endowed with the norm
\[
  \|f \|=\sup_{v \in T} |f'(v)|.
\]
Colonna and Easley showed in \cite{CoEa1} that $\Lip$, with an equivalent norm, is a
Banach space. We use the present norm, following \cite{CoMA1}, to make
some calculations cleaner.

The following lemma, which we will use later, can be easily obtained using the results in
\cite{CoEa1}. For the sake of completeness, we give the proof here (which is a slight
modification of the proof in \cite{CoEa1}) since we are using a
different norm.

\begin{lemma}\label{le:co_ea}
  Let $f \in \Lip$ and $v \in T$. Then
  \[
    |f(v)|\leq (|v|+1) \| f \|.
  \]
\end{lemma}
\begin{proof}
First, we claim that if $f \in \Lip$, $f(\root)=0$ and $\| f \| \leq
1$, then $|f(v)|\leq |v|$ for every $v \in T$. We prove this by
induction on $|v|$. It is clear that the claim is true for
$|v|=0$. Assume that the claim holds for all $v \in T$ with $|v|=n \in \N_0$, and let $w
\in T$ with $|w|=n+1$. Then
\[
  |f(w)| \leq |f(w)-f(\parent{w})| + |f(\parent{w})| =  |f'(w)| +
  |f(\parent{w})| \leq \| f \| + |\parent{w}| \leq 1 + n = |w|,
\]
which completes the induction step and finishes the proof of the claim.

Now, observe that the lemma is trivial if $f$ is the zero function. So
assume that $f$ is not identically zero. For the moment, assume $\|f
\|=1$. Define $g$ as $g(v)=f(v)-f(\root)$. Clearly, $g \in
\Lip$. Observe that $g'(v)=f'(v)$ for all $v \in T^*$ and
$g(\root)=g'(\root)=0$. Hence
\[
  \| g \| = \sup_{v\in T} |g'(v)| = \sup_{v\in T^*} |f'(v)| \leq \| f \| = 1,  
\]
so aplying the claim to $g$ we obtain
\[
|f(v)-f(\root)|=|g(v)|\leq |v|,
\]
for every $v \in T$. But from this we obtain
\[
|f(v)|\leq |f(v)-f(\root)|+ |f(\root)| \leq |v|+1,
\]
which proves the theorem for functions $f \in \Lip$ with $\| f
\|=1$. Now let $f$ be an arbitrary nonzero function in
$\Lip$. Applying the previous argument to $\frac{f}{\|f\|}$ we obtain
\[
\frac{|f(v)|}{\|f\|} \leq |v|+1,
\]
which finishes the proof.
\end{proof}

Also of interest is the {\em little Lipshitz space of $T$}, denoted by
$\Lip_0$, defined as the
set of all $f \in \calF$ for which
\[
  \lim_{|v|\to \infty} f'(v)=0.
\]
Clearly $\Lip_0$ is a subset of $\Lip$ and it can be shown
(see \cite{CoEa1}) that it is a separable closed subspace of $\Lip$.

In Section \ref{sec_hyper}, we will talk about hypercyclicity. Recall
that a bounded operator $A$ on a Banach space $\calB$ is hypercyclic if there
exists a vector $f \in \calB$ (called a hypercyclic vector for $A$) such that
the orbit of $f$ under $A$ is dense in the Banach space; i.e., the set
\[
\{ f, Af, A^2 f, A^3 f, \dots \}
\]
is dense in $\calB$. Clearly, if $A$ is hypercyclic, then $\calB$ must
be separable. Also observe that if $f$ is a hypercyclic vector, then
so is $A^n f$, for any $n \in \N$. Thus, if $A$ is hypercyclic, then
the set of its hypercyclic vectors is dense in $\calB$.

One way to prove
that an operator is hypercyclic is to apply the hypercyclicity
criterion. We include here the version we will use in this paper.

\begin{theorem}[Hypercyclicity Criterion]\label{th:hyp_cri}
Let $\calB$ be a separable Banach space and $A$ a bounded operator on
$\calB$. Assume there exists a set $X$, dense in $\calB$, and for each
$n \in \N$ there exists a function $R_n : X \to \calB$ such that, for every $f \in
X$ we have
\begin{itemize}
\item $A^n f \to 0$ as $n \to \infty$,
\item $R_n f \to 0$  as $n \to \infty$, and
\item $A^n R_n f \to f$  as $n \to \infty$.
\end{itemize}
Then $A$ is hypercyclic.
\end{theorem}

The proof (of a more general version) of this theorem can be found in
\cite[p.~74]{GEP}. A lot more information about the fascinating topic
of hypercyclicity can be found in \cite{GEP} and \cite{BaMa}.

\section{The forward shift and its adjoint}\label{sec_forward_shift}

We now present the two main objects of study in this note. The first
operator was originally defined, on $\ell^2(V)$, in \cite{JJS}.

\begin{defi}
  Let $T$ be a a rooted, countably infinite and localy finite
  tree. The {\em forward shift} operator $S: \calF \to \calF$ is defined as
  \[
    (Sf)(v) = \begin{cases}
      0, & \text{ if } v = \root, \\
      f(\parent{v}), & \text{ if } v \neq \root
    \end{cases}
  \]
\end{defi}

It is clear that $S$ is a linear operator on $\calF$. In
\cite[Proposition 3.4.1 (iii)]{JJS}, Jab{\l}onski, Jung and Stochel
found the form for the (Hilbert space) adjoint of $S$ on
$\ell^2(V)$. Inspired by this result, the first author made the
following definition in \cite{MA}.

\begin{defi}
  Let $T$ be a rooted, countably infinite and locally finite tree. The {\em backward shift} operator $B: \calF \to
  \calF$ is defined as
  \[
    (Bf)(v) = \sum_{w \in \Chi{v}} f(w),
  \]
  where if $v$ has no children, the sum is understood to be empty and hence $(Bf)(v)=0$.
\end{defi}

Also, it is clear that $B$ is a linear operator on $\calF$. We should
point out that the operator $S+B$ is the {\em adjacency operator (or
  adjacency matrix)} of the graph $T$. We will see at the end of this
section that $B$ is the Banach space adjoint of $S$ (on an appropriate
space). We start with some results about the forward shift.

\begin{theorem}\label{th:isometryL}
Let $T$ be a rooted, countable infinite and locally finite tree and let $S$ be the
forward shift. Then $S:\Lip \to \Lip$ is bounded and $\|S \| \leq 1$. If $T$ is leafless, then $\| S f
\|= \|f \|$. 
\end{theorem}
\begin{proof}
Let $f \in \Lip$. It is a straightforward calculation to check that
$(S f)'=S f'$. Hence we get
\[
\| S f \| = \sup_{v \in T} |(Sf)'(v)| = \sup_{v \in T} |(S f')(v)| =
\sup_{v \in T^*} |(S f')(v)| =  \sup_{v \in T^*} |f'(\parent{v})| \leq 
\sup_{w \in T} |f'(w)| = \| f \|,
\]
and hence $\|S\| \leq 1$. If $T$ is leafless, the 
the inequality in the expression above is an equality. Hence $S$ is an isometry, as desired. 
\end{proof}

The same result holds for $S$ as an operator on $\Lip_0$.

\begin{theorem}\label{th:isometryL0}
Let $T$ be a rooted, countable infinite and locally finite tree and let $S$ be the
forward shift. Then $S:\Lip_0 \to \Lip_0$ is bounded 
and $\|S \| \leq 1$. If $T$ is leafless, then $\| S f
\|= \|f \|$. 
\end{theorem}
\begin{proof}
We first observe that if $f \in \Lip_0$, then $Sf \in \Lip_0$. Indeed, 
\[
\lim_{|v|\to \infty} (Sf)'(v)= \lim_{|v|\to \infty} (Sf')(v)=
\lim_{|v|\to \infty} f'(\parent{v})= \lim_{|w|\to \infty} f'(w)=0.
\]
Since $\Lip_0$ is a closed subspace of $\Lip$, it follows that $S$ is
a bounded operator on $\Lip_0$. The calculation in the theorem above then
shows that $\|S\| \leq 1$ and that $S$ is an isometry if $T$ is leafless.
\end{proof}

If $T$ has leaves, we will show later (right after Proposition~\ref{prop:empty_spec}) that $S$ has nontrivial kernel (both as an operator on $\Lip$ and on $\Lip_0$) and hence it is not an isometry.

If $T$ is leafless, we have shown that $S$ is an isometry on $\Lip$. By a
theorem of Koehler and Rosenthal (see, for example,
\cite[p.~11]{FlJa}) there exists a semi-inner product $[\cdot,\cdot]$
in $\Lip$ such that $[Sf,Sg]=[f,g]$ for all $f, g \in \Lip$. This fact
raises a lot of questions about the structure of $\Lip$ with this
seminorm. We plan to explore this in future research.

On the other hand, if $T$ has leaves, one may ask how ``far'' is $S$
from an isometry.\footnote{We thank a referee for suggesting this
  question and for providing the reference in the previous paragraph.} It is not hard to see (for example, use
Theorem~\ref{th:specS} below and mimic the
argument in \cite[Problem 150]{halmos}) that $\|S-U\|=2$, if $U$ is a
surjective isometry. Is it possible that there exists a (non
surjective) isometry $U$ such that $\|S-U\|< 2$? We leave this
question open for future research.

We should point out that if $T$ has leaves, we can define
\[
  \calM:=\{f \in \Lip \, : \, f(v)=0 \text{ for all } v \in A \},
\]
where $A:=\{ w \in T: w \text{ is a leaf or } w=\nparent{v}{n} \text{ for some leaf } v \in T \text{
  and } n \in \N \}$; i.e., $\calM$ is the subspace of all functions
that vanish on every leaf and every ancestor of a leaf. Then clearly
$\calM$ is an invariant subspace for $S$ (of finite codimension if $T$ has
finitely many leaves) and $S \big|_\calM$ is an isometry. Thus, even
if $T$ has leaves, there is a subspace in which $S$ acts isometrically.

We will now study the spectrum of $S$. We first show that the forward
shift has no eigenvalues, not even on $\calF$. Recall that for an
operator $A$, the set of eigenvalues, the approximate point spectrum and the
spectrum, are denoted by $\sigma_{\operatorname{p}}(A)$,
$\sigma_{\operatorname{ap}}(A)$,  and $\sigma(A)$, respectively. (The
relevant definitions can be found in, e.g., \cite{Conway}.)

\begin{proposition}\label{prop:empty_spec}
Let $T$ be a rooted, leafless, countably infinite and locally finite tree and let $S$ be the
forward shift on $\calF$. Then $\sigma_{\operatorname{p}}(S)=\varnothing$.
\end{proposition}
\begin{proof}
Let $f \in \calF$. Clearly, $(Sf)(v)=0$ for all $v \in T$ implies that
$f(\parent{v})=0$ for all $v \in T^*$ and hence $f(w)=0$ for all $w
\in T$. Thus $\lambda=0$ is not eigenvalue.

Assume there exists $\lambda \neq 0$ such that $Sf=\lambda f$, with $f
\in \calF$. Since $0=(Sf)(\root)=\lambda f(\root)$ it follows that
$f(\root)=0$. Now, for every $v \in \Chi{\root}$ we have $\lambda f(v) =
(Sf)(v)=f(\parent{v})=f(\root)=0$ and hence $f(v)=0$. Continuing in
this manner, we obtain that $f(v)=0$ for all $v \in T$ and thus $\lambda$ is not an eigenvalue.
\end{proof}

Observe that if the tree has a leaf, then $0$ is an eigenvalue of $S$ as an operator on $\calF$:
indeed, if $v$ is a leaf, then $S \chi_{\{v\}}=0$, where
$\chi_{\{v\}}$ is the characteristic function of $v$. The proof above
shows that, in this case, $0$ is the unique eigenvalue of $S$.

As a corollary, it should be noted that in the leafless case, $S$ has
no eigenvalues as an operator on $\Lip$ and as an operator on
$\Lip_0$. If $T$ has a leaf, $0$ is an eigenvalue for $S$ both on
$\Lip$ and on $\Lip_0$, with eigenvector $\chi_{\{v\}}$. (By the way,
this also shows that if $T$ has a leaf, then $S$ is not an isometry on
$\Lip$ nor on $\Lip_0$; see Theorems \ref{th:isometryL} and \ref{th:isometryL0}).

In the leafless case, since $S$ is an isometry, its approximate point spectrum lies in the
unit circle. Indeed, let $\lambda \in \C$ with $|\lambda|\neq 1$; then
\[
  \| (S-\lambda) f \| \geq \mod{ \|Sf\| - \| \lambda f \| } =
  \mod{ 1-|\lambda| } \,  \|  f \|,
\]
and hence $S-\lambda$ is bounded below. Therefore, $\lambda \notin
\sigma_{\operatorname{ap}}(S)$. 

The following theorem gives a full description of the spectrum of
$S$ in the leafless case. It should be noted that it is known that the spectrum of a
noninvertible isometry is always $\overline{\D}$
(e.g. \cite[p.~213]{Conway}), we prefer to give an independent proof
since it gives more information about $S$.

\begin{theorem}\label{th:specS}
Let $T$ be a rooted, countable infinite and locally finite tree and let $S$ be the
forward shift on $\Lip$ or on $\Lip_0$. Then,
$\sigma(S)=\overline{\D}$. If $T$ is leafless, then $\sigma_{\operatorname{ap}}(S)=\partial \D$.
\end{theorem}
\begin{proof}
  Since $\|S\|\leq 1$, it follows that $\sigma(S) \subseteq
  \overline{\D}$. It is easily verified that the equation $Sf =
  \chi_{\{\root\}}$ has no solution $f \in \calF$ (just evaluate at
  the root $\root$), hence $0 \in \sigma(S)$.

  Let $\lambda \in \D$, $\lambda \neq 0$. Then the equation
  \[
    (S-\lambda) f = \chi_{\{\root\}}
  \]
  has a unique solution $f \in \calF$ given by $f(v)= \frac{-1}{\lambda^{|v|+1}}$, as a
  straightforward calculation shows. But in this case, $f'(v)$, for $v
  \in T^*$, is given by
  \[
    f'(v)=(1-\lambda^{-1}) \lambda^{-|v|}.
  \]
  But since $\lambda \in \D$, the function $f'$ is unbounded, and thus
  $f \notin \Lip$ (and $f \notin \Lip_0$). Hence $S-\lambda$ is not
  surjective and thus $\D \setminus \{ 0 \} \subseteq \sigma(S)$. It
  then follows that $\overline{\D} \subseteq  \sigma(S)$ and hence
  $\sigma(S)=\overline{\D}$.

  Lastly, assume $T$ is leafless. Recall that for any operator $A$ we have $\partial \sigma(A) \subseteq
  \sigma_{\operatorname{ap}}(A)$ (e.g. \cite[Prop. 6.7]{Conway}
  and hence $\partial \D = \partial \sigma(S) \subseteq \sigma_{\operatorname{ap}}(S)
  \subseteq \partial \D$ (since $S$ is an isometry, as noted above). Therefore $
  \sigma_{\operatorname{ap}}(S)= \partial \D$, as desired.
\end{proof}
If $T$ has leaves, as we showed before, $0$ is an eigenvalue of $S$
and thus $  \sigma_{\operatorname{ap}}(S) \neq  \partial \D$. Also
observe that, as a corollary, we obtain that $\|S\|=1$, even in the
case where $T$ has leaves.

In \cite{CoMA1} it is shown that the dual space of $\Lip_0$ is
(isometrically isomorphic to) the space $L^1(T)$ and the dual space of
$L^1(T)$ is (isometrically isomorphic to) $\Lip$. Using the
identification in \cite{CoMA1} we can make the following observations.

\begin{proposition}
Let $T$ be a rooted, countable infinite and locally finite tree. If $S:\Lip_0 \to
\Lip_0$ is the forward shift, then $S^*: L^1(T) \to L^1(T)$ is given
by $S^*=B$, where $B$ is the backward shift restricted to $L^1(T)$.
\end{proposition}
\begin{proof}
It is shown in \cite{CoMA1} that, for every $f \in L^1(T)$, the
functional $\Phi_f: \Lip_0 \to \C$, defined as
\[
  \Phi_f(g)=\sum_{v \in T} f(v) g'(v),
\]
for each $g \in \Lip_0$, is bounded and the mapping $f \mapsto \Phi_f$
is an isometric isomorphism from $L^1(T)$ onto $\Lip_0^*$. Using this
identification we obtain
\begin{equation}\label{eq:duality}
\begin{split}
  (S^* \Phi_f)(g)
  &= \Phi_f(Sg) \\
  &= \sum_{w\in T} f(w) (Sg)'(w) \\
  &= f(\root) (Sg) (\root) + \sum_{w\in T^*} f(w) \left( (Sg)(w) - (Sg)(\parent{w}) \right) \\
  &= \sum_{v\in T^*} f(w) \left( (Sg)(w) - (Sg)(\parent{w} \right)\\
  &= \sum_{w\in T^{*}} f(w) g(\parent{w}) - \sum_{w\in T^{**}} f(w) g(\nparent{w}{2}).
\end{split}
\end{equation}
Observe that for every vertex $v \in T$ either $\Chi{v}$ is empty and hence 
\[
\sum_{w\in \Chi{v}} f(w)=0,
\]
or there are vertices $w \in T^*$ with $v=\parent{w}$. Therefore,
\[
\sum_{w\in T^{*}} f(w) g(\parent{w})=\sum_{v \in T} g(v) \left( \sum_{w\in \Chi{v}} f(w)\right).
\]
Hence, Equation \eqref{eq:duality} implies that
    \begin{align*} 
      (S^* \Phi_f)(g)
  &= \sum_{v \in T} g(v) \left( \sum_{w\in \Chi{v}} f(w)\right) - \sum_{v \in T^*}
      g(\parent{v}) \left( \sum_{w\in \Chi{v}} f(w) \right) \\
  &= \left( \sum_{w\in \Chi{\root}} f(w)\right) g(\root) + 
      \sum_{v \in T^*} \left( \sum_{w\in \Chi{v}} f(w)\right) (g(v)-
      g(\parent{v}))\\
  &= \sum_{v \in T} \left( \sum_{w\in \Chi{v}} f(w)\right) g'(v) \\
  &= \sum_{v \in T} (Bf)(v) g'(v) \\
  &= \Phi_{Bf}(g).
\end{align*}
Hence $S^*$ can be identified with $B$ on $L^1(T)$.
\end{proof}

Using the same technique as above, the following result can be shown.
\begin{proposition}
Let $T$ be a rooted, countable infinite and locally finite tree. If $B:L^1(T) \to
L^1(T)$ is the backward shift, then $B^*: \Lip \to \Lip$ is given by $B^*=S$.
\end{proposition}

The above results show why we choose to call $B$ the ``backward''
shift, as an analogy of what happens with the classical forward and
backward shifts on $\ell^p(\N)$. We study this operator in greater
depth in the next section.

\section{The backward shift}\label{sec_backward_shift}

It will turn out that the backward shift operator is not always
bounded on $\Lip$ or on $\Lip_0$, as was the case for the forward
shift. We will need the following definition (see Figure~\ref{fig:homogeneous} for an example).

\begin{defi}
Let $T$ be a rooted, countably infinite and locally finite tree. We say that
$T$ is {\em homogenous by sectors (at the level $N$)} if there exists $N \in \N_0$ such
that for all $v \in T$ with $|v|=N$, we have $\gamma(v)=\gamma(w)$ for
each $w \in S_v$.
\end{defi}

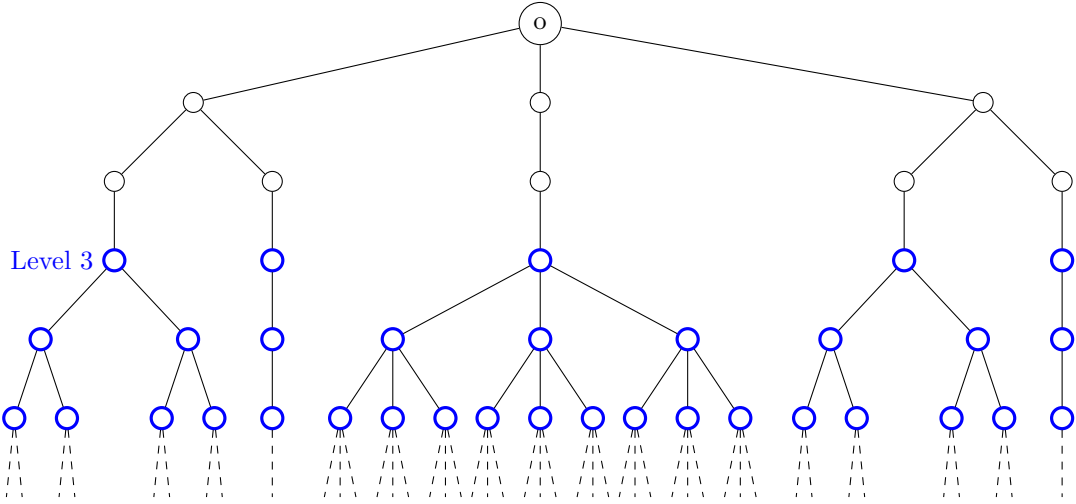
\begin{figure}[!htbp]
\centering
\begin{tikzpicture}[level 1/.style={sibling distance=7.5cm}, level
  2/.style={sibling distance=3cm}, level 3/.style={sibling
    distance=1.5cm}, level 4/.style={sibling distance=2.8cm}, level
  5/.style={sibling distance=1cm}, level 6/.style={sibling
    distance=0.3cm}, emph/.style={edge from parent/.style={black, dashed,draw}},
  scale=0.7]
\node[draw,circle, scale=1] {$\root$} [grow'=down] 
child {node[draw,circle, scale=.8,right=0.5cm]  {} 
    child{node[draw,circle, scale=.8] {}
         child{node[arn_r, draw,circle, scale=.8] {} 
              child{node[arn_r, draw,circle, scale=.8]{}
                child{node[arn_r, draw,circle, scale=.8]{}
                child[emph]{
                }
              }
            }
          }
        }
        child{node[draw,circle, scale=.8] {}
           child{node[arn_r, draw,circle, scale=.8] {}
             child{node[arn_r, draw,circle, scale=.8] {}
               child{node[arn_r, draw,circle, scale=.8] {}
                 child[emph]{}
               child[emph]{}}
             child{node[arn_r, draw,circle, scale=.8] {}
               child[emph]{}
             child[emph]{}}                 
                 }
             child{node[arn_r, draw,circle,
               scale=.8] {}
               child{node[arn_r, draw,circle, scale=.8] {}
                 child[emph]{}
               child[emph]{}}
               child{node[arn_r, draw,circle, scale=.8] {}
                 child[emph]{}
               child[emph]{}}                 
               }
            }}}
child {node[draw,circle, scale=.8]  {} 
    child {node[draw,circle, scale=.8]{}
           child{node[arn_r, draw,circle, scale=.8] {}
              child{node[arn_r, draw,circle, scale=.8] {}
                child{node[arn_r, draw,circle, scale=.8] {}
                 child[emph]{}
                 child[emph]{}
                 child[emph]{}}
               child{node[arn_r, draw,circle, scale=.8] {}
               child[emph]{}
                 child[emph]{}
                 child[emph]{}}
               child{node[arn_r, draw,circle, scale=.8] {}
               child[emph]{}
                 child[emph]{}
                 child[emph]{}}                 
                 }
              child{node[arn_r, draw,circle, scale=.8] {}
                child{node[arn_r, draw,circle, scale=.8] {}
                child[emph]{}
                 child[emph]{}
                 child[emph]{}}
               child{node[arn_r, draw,circle, scale=.8] {}
               child[emph]{}
                 child[emph]{}
                 child[emph]{}}
               child{node[arn_r, draw,circle, scale=.8] {}
               child[emph]{}
                 child[emph]{}
                 child[emph]{}}                  
               }
             child{node[arn_r, draw,circle, scale=.8] {}
               child{node[arn_r, draw,circle, scale=.8] {}
               child[emph]{}
                 child[emph]{}
                 child[emph]{}}
               child{node[arn_r, draw,circle, scale=.8] {}
               child[emph]{}
                 child[emph]{}
                 child[emph]{}}
               child{node[arn_r, draw,circle, scale=.8] {}
               child[emph]{}
                 child[emph]{}
                 child[emph]{}}                  
               }
            }
         }
       }
child {node[draw,circle, scale=.8,right=0.5cm]  {}
  child{node[draw,circle, scale=.8] {}
    child{node[arn_r, draw,circle, scale=.8] {} 
              child{node[arn_r, draw,circle, scale=.8]{}
                child{node[arn_r, draw,circle, scale=.8]{}
                child[emph]{
                }
              }
            }
          }
          }
    child{node[draw,circle, scale=.8] {} 
        child{node[arn_r, draw,circle, scale=.8,label=left:{\textcolor{blue}{Level $3$}}] {}
             child{node[arn_r, draw,circle, scale=.8] {}
               child{node[arn_r, draw,circle, scale=.8] {}
               child[emph]{}
               child[emph]{}}
             child{node[arn_r, draw,circle, scale=.8] {}
             child[emph]{}
               child[emph]{}}                 
                 }
             child{node[arn_r, draw,circle,
               scale=.8] {}
               child{node[arn_r, draw,circle, scale=.8] {}
                 child[emph]{}
               child[emph]{}}
             child{node[arn_r, draw,circle, scale=.8] {}
             child[emph]{}
               child[emph]{}}                 
            }
         }
       }
     };
 \end{tikzpicture}
\caption{A tree which is homogeneous by sectors at the level $3$.\label{fig:homogeneous}}
\end{figure}

Intuitively, a tree is homogeneous by sectors if after some level
every subtree is a homogeneous tree. For a tree $T$, homogeneous by sectors at the level $N$, we define
\[
  \Lambda:=\sup\{\gamma(v) +\gamma(\parent{v}) -1 + |\gamma(v)-1|+ |\gamma(v)-\gamma(\parent{v})| \, |v|  \, : \, v \in T^* \}.
\]
Since the tree is homogeneous by sectors at the level $N$, it turns out the above supremum is finite and, in fact,
\[
\Lambda = \max \{ \gamma(v) +\gamma(\parent{v}) -1 + |\gamma(v)-1|+ |\gamma(v)-\gamma(\parent{v})| \, |v|  \, : \, 0<|v|\leq N \}
\]
As it turns out, $B$ is bounded if the tree is homogeneous by sectors and $\Lambda$ gives an estimate of the norm of $B$. In the corollary that follows the next proposition, we will find a more manageable estimate.

\begin{proposition}\label{prop:boundB}
Let $T$ be a rooted, countable infinite and locally finite tree. Assume that $T$
is homogeneous by sectors at the level $N$. Then $B$ is bounded on
$\Lip$. Furthermore, 
\[
  \|B \| \leq \max\{ 2\gamma(\root), \Lambda \}.
\]
\end{proposition}
\begin{proof}
  First observe that
  \[
    (Bf)'(\root) = (Bf)(\root)= \sum_{w \in \Chi{\root}} f(w) = \sum_{w \in
      \Chi{\root}} (f(w)-f(\parent{w})) + \gamma(\root) f(\root) = \sum_{w \in
      \Chi{\root}} f'(w) + \gamma(\root) f(\root)
  \]
  and hence, since $|f'(w)|\leq \| f \|$ for all $w \in T$, we have
  \[
    |(B f)'(\root)| \leq \gamma(\root) \| f \| +  \gamma(\root)
    |f(\root)| \leq 2 \gamma(\root) \, \| f \|.
  \]
Now, for $v \in T^*$, we have
\begin{equation}\label{eq:bf'_estimate}
  \begin{split}
  (Bf)'(v)
  &= \sum_{w\in \Chi{v}} f(w) - \sum_{w \in \Chi{\parent{v}}} f(w) \\
  &= \sum_{w\in \Chi{v}} \left(f(w)-f(\parent{w})\right) + \gamma(v) f(v) - \sum_{w
      \in \Chi{\parent{v}}} \left(f(w)-f(\parent{w})\right) 
      -\gamma(\parent{v}) f(\parent{v}) \\
  &= \sum_{w\in \Chi{v}} f'(w)  - \sum_{w \in
    \Chi{\parent{v}}} f'(w) + \gamma(v) f(v)
    -\gamma(\parent{v}) f(\parent{v}),
  \end{split}
  \end{equation}
and hence
\begin{align*}
  (Bf)'(v)
  &= \sum_{w\in \Chi{v}} f'(w) - \sum_{w \in
    \Chi{\parent{v}}} f'(w) +  \gamma(v) f'(v)
  +(\gamma(v)-\gamma(\parent{v})) f(\parent{v})\\
    &= \sum_{w\in \Chi{v}} f'(w) - \sum_{\substack{w \in
    \Chi{\parent{v}} \\ w\neq v}} f'(w) + ( \gamma(v)-1) f'(v)
  +(\gamma(v) - \gamma(\parent{v})) f(\parent{v})
\end{align*}

Recalling that $|f'(v)|\leq \| f \|$ for all $v \in T^*$, it follows that
\[
  |(Bf)'(v)| \leq \gamma(v) \|f\| +(\gamma(\parent{v})-1) \|f \| + |\gamma(v)-1| \, \| f \| + |\gamma(v)-\gamma(\parent{v})| \, |f(\parent{v})|.
\]

Since $|f(\parent{v})|\leq (|\parent{v}|+1) \|f \|$ by Lemma \ref{le:co_ea}, we obtain
\[
  |(Bf)'(v)| \leq (\gamma(v) + \gamma(\parent{v})-1+|\gamma(v)-1|) \|f \| + |\gamma(v)-\gamma(\parent{v})| \, |v| \, \|f \|.
\]
Therefore, $|(Bf)'(v)| \leq \Lambda \|f \|$, for all $v \in T^*$. Since $|(Bf)'(\root)| \leq 2 \gamma(\root) \|f \|$, it follows that
\[
  \|B \| \leq \max\{ 2\gamma(\root), \Lambda \},
\]
as desired.
\end{proof}


We now provide a more computable estimate of the norm of the backward
shift. For a tree $T$, homogeneous by sectors at the level $N$, we define
\[
  \Gamma:=\sup\{ \gamma(v) \, : \, v \in T \}=\max \{ \gamma(v) \, :
  \, 0 \leq |v|\leq N \}
\]
and
\[
  \Omega:= \max\{ |(\gamma(v)-\gamma(\parent{v})| \:  |v| \, : \, 0<|v|\leq N \}.
\]
Observe that if $T$ is homogeneous of order $\gamma$, then
$\Gamma=\gamma$ and $\Omega=0$.

The number $\Lambda$ that we defined above can be estimated using
$\Gamma$ and $\Omega$. Indeed, if $0<|v|\leq N$ and  $\gamma(v) \geq 1$, then
\begin{align*}
\gamma(v)+\gamma(\parent{v}) -1+|\gamma(v)-1| +
  |(\gamma(v)-\gamma(\parent{v})| \, |v|
  & =  2\gamma(v)+\gamma(\parent{v}) -2 +
    |(\gamma(v)-\gamma(\parent{v})| \, |v| \\
  &\leq 3\Gamma-2+\Omega.
\end{align*}
If $0<|v|\leq N$ and  $\gamma(v)=0$,
\begin{align*}
\gamma(v)+\gamma(\parent{v}) -1+|\gamma(v)-1| +
  |(\gamma(v)-\gamma(\parent{v})| \, |v|
  & = \gamma(\parent{v}) +  |(\gamma(v)-\gamma(\parent{v})| \, |v|\\
  &\leq \Gamma +\Omega \\
  &\leq 3\Gamma-2+\Omega.
\end{align*}
Hence, $\Lambda \leq 3\Gamma-2+\Omega$ and we obtain the following corollary.

\begin{corollary}\label{cor:boundB}
Let $T$ be a rooted, countable infinite and locally finite tree. Assume that $T$
is homogeneous by sectors at the level $N$. Then $B$ is bounded on
$\Lip$. Furthermore, 
\[
  \|B \| \leq \max\{ 2 \gamma(\root), 3 \Gamma -2 + \Omega\}.
\]
\end{corollary}


Observe that $3 \Gamma -2 + \Omega < 2 \gamma(\root)$ if and only if $3 \Gamma - 2 \gamma(\root) + \Omega < 2$ and this occurs if and only if $\gamma(\root)=\Gamma=1$ and
$\Omega=0$. That is, if and only if $T$ is a homogeneous tree of order
$1$. We obtain the following corollary.

\begin{corollary}\label{cor:boundB_L0}
Let $T$ be a rooted, countable infinite and locally finite tree. Assume that $T$
is homogeneous by sectors at the level $N$. If $T$ is homogeneous of
order $1$ then $B$ is bounded on $\Lip$ and $\| B \|
\leq 2$. In any other case, $B$ is bounded on $\Lip$
and $\|B \| \leq  3 \Gamma -2 + \Omega$.
\end{corollary}

The result in the theorem above also holds for the little Lipschitz space.

\begin{theorem}
Let $T$ be a rooted, countable infinite and locally finite tree. Assume that $T$
is homogeneous by sectors at the level $N$. Then $B$ is bounded on $\Lip_0$.
\end{theorem}
\begin{proof}
Since $\Lip_0$ is a closed subspace of $\Lip$, it is enough to show
that, if $f \in \Lip_0$, then $Bf \in \Lip_0$.

Let $\epsilon>0$. Since $f \in \Lip_0$, there exists $N_1$ such that
\[
  |f'(v)|<\frac{\epsilon}{3 \Gamma}, \quad \text{ if } \quad |v|\geq N_1.
\]
By Equation \eqref{eq:bf'_estimate}, for $v \in T^*$ we have
\[
  (Bf)'(v) = \sum_{w\in \Chi{v}} f'(w)  -\sum_{w \in \Chi{\parent{v}}}
  f'(w) + \gamma(v) f(v) -\gamma(\parent{v}) f(\parent{v}).
\]
Hence, since $\gamma(v)=\gamma(\parent{v})$ for every $v\in T^*$ with
$|v|\geq N+1$, we have that if $|v|\geq \max\{N_1,N+1\}$ then
\begin{align*}
  |(Bf)'(v)|
  &\leq \Gamma \frac{\epsilon}{3 \Gamma} + \Gamma \frac{\epsilon}{3
    \Gamma} +   |\gamma(v) f(v) -\gamma(\parent{v}) f(\parent{v})| \\
  &\leq \frac23 \epsilon + \gamma(v) |f'(v)|\\
    &\leq \frac23 \epsilon + \Gamma |f'(v)|\\
  & < \epsilon.
\end{align*}
 Hence $Bf \in \Lip_0$, as desired.
\end{proof}

It is clear that the norm of $B$, as an operator on $\Lip_0$,
satisfies the same estimates as in Proposition~\ref{prop:boundB} and
Corollary~\ref{cor:boundB_L0}.

We want to prove the converse of Proposition~\ref{prop:boundB}. For
that, we need the following lemma.

\begin{lemma}\label{le:sup_gamma}
Let $T$ be a rooted, countable infinite and locally finite tree. If $\sup\{
\gamma(v) \, : \, v \in T \}=\infty$, then for every $n\in \N$ there
exists $v_n \in T$ such that $|v_n|\geq n$ and $\gamma(\parent{v_n})<\gamma(v_n)$.
\end{lemma}
\begin{proof}
By contradiction, assume that there exists $N \in \N$ such that for
all $v \in T$ with $|v|\geq N$ we have $\gamma(v) \leq
\gamma(\parent{v})$. It follows that, for all $k \in \N$, we have that
if $|v|=N+k$, then
\[
  \gamma(v) \leq \gamma(\parent{v}) \leq \gamma(\nparent{v}{2}) \leq
  \gamma(\nparent{v}{3}) \leq \cdots  \leq \gamma(\nparent{v}{k}).  
\]
But, since $|\nparent{v}{k}|=N$, this implies that
\[
  \sup\{ \gamma(v) \, : \, |v|\geq N\} \leq \max\{ \gamma(u) \, : \,
  |u|=N \} 
\]
and hence that $\sup\{ \gamma(v) \, : \, v \in T\} \leq \max\{
\gamma(v) \, : \, |v| \leq N \} <\infty$, contradicting the hypothesis.
\end{proof}

We can now show that homogeneity by sectors is actually a necessary
condition for boundedness of $B$.

\begin{theorem}
Let $T$ be a rooted, countable infinite and locally finite tree. If $B: \Lip
\to \Lip$ is bounded, then $T$ is homogeneous by sectors.
\end{theorem}
\begin{proof}
Consider the function $g \in \Lip$ given by $g(v)=|v|$. Observe that,
for every $v \in T^*$ we have
\begin{equation}\label{eq:bg'_gamma}
  \begin{split}
  (Bg)'(v)
  &=\sum_{w \in \Chi{v}} |w| - \sum_{w \in  \Chi{\parent{v}}} |w|
  \\
  &= \gamma(v) (|v|+1)- \gamma(\parent{v})(|\parent{v}|+1) \\
  &= \left(\gamma(v) - \gamma(\parent{v}) \right)|v| +
    \gamma(v).
  \end{split}
\end{equation}

First we show that $\sup\{ \gamma(v) \, : \, v \in T
\}< \infty$. By contradiction, assume that $\sup\{ \gamma(v) \, : \, v \in T
\}=\infty$. By the previous lemma, there exists a sequence $(v_n)$ in $T$ such that
$|v_n|\geq n$ and $\gamma(\parent{v_n}) < \gamma(v_n)$. But then,
Equation \eqref{eq:bg'_gamma} gives
\[
  (Bg)'(v_n) = \left(\gamma(v_n) - \gamma(\parent{v_n}) \right)|v_n| +
    \gamma(v_n) \geq \left(\gamma(v_n) - \gamma(\parent{v_n}) \right)|v_n| \geq n,
\]
so $(Bg)'$ is unbounded and hence $Bg \notin \Lip$ contradicting the
boundedness of $B$. Hence $\sup\{ \gamma(v) \, : \, v \in T \}<
\infty$, as claimed.

Now, let $\Gamma=\sup\{ \gamma(v) \, : \, v \in T \}$. If $T$ is not
homogeneous by sectors then, for every $m \in \N$, there exists $v_m
\in T$ with $|v_m|\geq m$ such that $\gamma(v_m) \neq
\gamma(\parent{v_m})$. But then Equation \eqref{eq:bg'_gamma} implies
\begin{align*}
   |(Bg)'(v_m)|
  &= |\left(\gamma(v_m) - \gamma(\parent{v_m}) \right)|v_m| +
    \gamma(v_m)| \\
  &\geq  \mod{ \gamma(v_m) - \gamma(\parent{v_m}) } \, |v_m|
    - |\gamma(v_m)| \\
  &\geq |v_m| - \gamma(v_m) \\
  &\geq m -\Gamma,
\end{align*}
which implies that $(Bg)'$ is unbounded and hence $Bg \notin \Lip$ contradicting the
boundedness of $B$. Therefore, $T$ must be homogeneous by sectors.
\end{proof}

A similar result holds for $\Lip_0$, with basically the same
proof. We include the details for the sake of completeness.
\begin{theorem}
Let $T$ be a rooted, countable infinite and locally finite tree. If $B: \Lip_0
\to \Lip_0$ is bounded, then $T$ is homogeneous by sectors.
\end{theorem}
\begin{proof}
Consider the function $g \in \Lip_0$ given by $\ds
g(v)=\sum_{k=1}^{|v|} \frac{1}{k}$. Observe that, for every $v \in T^*$ we have
\begin{equation}\label{eq:bg'_gamma2}
  \begin{split}
  (Bg)'(v)
  &=\sum_{w \in \Chi{v}} \sum_{k=1}^{|w|} \frac{1}{k} - \sum_{w \in
    \Chi{\parent{v}}} \sum_{k=1}^{|\parent{w}|}\frac{1}{k} \\
  &=\gamma(v) \sum_{k=1}^{|v|+1} \frac{1}{k} - \gamma(\parent{v})
  \sum_{k=1}^{|\parent{v}|+1}\frac{1}{k} \\
  &= \left(\gamma(v) - \gamma(\parent{v}) \right) \sum_{k=1}^{|v|} \frac{1}{k} +
    \frac{\gamma(v)}{|v|+1}.
  \end{split}
\end{equation}

First we show that $\sup\{ \gamma(v) \, : \, v \in T
\}< \infty$. By contradiction, assume that $\sup\{ \gamma(v) \, : \, v \in T
\}=\infty$. By Lemma \ref{le:sup_gamma}, there exists a sequence $(v_n)$ in $T$ such that
$|v_n|\geq n$ and $\gamma(\parent{v_n}) < \gamma(v_n)$. But then,
Equation \eqref{eq:bg'_gamma2} gives
\[
  (Bg)'(v_n) =  \left(\gamma(v_n) - \gamma(\parent{v_n}) \right) \sum_{k=1}^{|v_n|} \frac{1}{k} +
    \frac{\gamma(v_n)}{|v_n|+1} \geq \left(\gamma(v_n) - \gamma(\parent{v_n})
    \right) \sum_{k=1}^{|v_n|} \frac{1}{k} \geq \sum_{k=1}^{n} \frac{1}{k}.
  \]
But this expression is unbounded and so is $(Bg)'$. Hence $Bg \notin \Lip_0$ contradicting the
boundedness of $B$.

Now, let $\Gamma=\sup\{ \gamma(v) \, : \, v \in T \}$. If $T$ is not
homogeneous by sectors then, for every $m \in \N$, there exists $v_m
\in T$ with $|v_m|\geq m$ such that $\gamma(v_m) \neq
\gamma(\parent{v_m})$. But then Equation \eqref{eq:bg'_gamma} implies
\begin{align*}
  |(Bg)'(v_m)|
  &= \mod{ \left(\gamma(v_m) - \gamma(\parent{v_m}) \right) \sum_{k=1}^{|v_m|} \frac{1}{k} +
    \frac{\gamma(v_m)}{|v_n|+1} } \\
  &\geq \mod{ \left(\gamma(v_m) - \gamma(\parent{v_m}) \right)
    \sum_{k=1}^{|v_m|} \frac{1}{k} } -
    \mod{ \frac{\gamma(v_m)}{|v_m|+1} } \\
  &\geq  \sum_{k=1}^{|v_m|} \frac{1}{k} -\frac{\Gamma}{|v_m|+1} \\
  &\geq \sum_{k=1}^{m} \frac{1}{k} -\Gamma, 
\end{align*}
which implies that $(Bg)'$ is unbounded and hence $Bg \notin \Lip_0$ contradicting the
boundedness of $B$. Therefore, $T$ must be homogeneous by sectors.
\end{proof}

The following proposition characterizes trees that are homogeneous by
sectors in terms of a combinatorial quantity.

\begin{proposition}
Let $T$ be a rooted, countable infinite and locally finite tree. Then $\sup \{
|\gamma(v)-\gamma(\parent{v})| \, |v| \, : \, v \in T^* \} < \infty$ if
and only if $T$ is homogeneous by sectors.
\end{proposition}
\begin{proof}
  First assume that $T$ is homogeneous by sectors at the level
  $N$. Then, for all $|v|>N$, we have $\gamma(v)=\gamma(\parent{v})$
  and hence
  \[
    \{ |\gamma(v)-\gamma(\parent{v})| \, |v| \, : \, v \in T^* \} =
    \{ |\gamma(v)-\gamma(\parent{v})| \, |v| \, : \, 0 < |v| \leq
    N \} \cup \{ 0 \}
  \]
  which is a finite set. Hence,  $\sup \{
  |\gamma(v)-\gamma(\parent{v})| \, |v| \, : \, v \in T^* \} < \infty$.

  Now, assume that  $\sup \{ |\gamma(v)-\gamma(\parent{v})| \, |v| \,
  : \, v \in T^* \} < \infty$. If $T$ was not homogeneous by sectors,
  for every $m \in \N$, there would exist $v_m \in T$ with $|v_m|\geq
  m$ and such that $\gamma(v_m)\neq \gamma(\parent{v_m})$. But then
  \[
    \mod{ \gamma(v_m)-\gamma(\parent{v_m}) } \, |v_m| \geq |v_m|
    \geq m
  \]
  and therefore $\sup \{|\gamma(v)-\gamma(\parent{v})| \, |v| \, : \,
  v \in T^* \}=\infty$, which is a contradition. Hence $T$ is homogeneous by sectors.
\end{proof}

We can summarize the results of this section in the following theorem.
\begin{theorem}\label{th:main}
Let $T$ be a rooted, countable infinite and locally finite tree. Let $B$ be the
backward shift. The following are equivalent.
\begin{itemize}
  \item $B: \Lip \to \Lip$ is bounded.
  \item $B: \Lip_0 \to \Lip_0$ is bounded.
  \item $T$ is homogeneous by sectors.
  \item $\sup \{ |\gamma(v)-\gamma(\parent{v})| \, |v| \, : \, v \in
      T^* \} < \infty$.
    \end{itemize}
    Moreover, in this case, $\|B \| \leq \max\{2\gamma(\root), \Lambda \}$.
\end{theorem}

We can obtain some simple estimates from below for the norm of the
backward shift. Indeed, assume $T$ is homogeneous by sectors and let
$g_1:T \to \C$ be defined as (see Figure~\ref{fig:h1} for an example)
\[
  g_1(v)=  \begin{cases}
    1, & \text{ if } v = \root, \\
    2, & \text{ if } v \in \Chi{\root},\\
    1, & \text{ if } v \in \nChi{\root}{2},\\
    0, & \text{ in any other case.}
  \end{cases}
\]

\begin{figure}[!htbp]
\centering
\begin{tikzpicture}[level 1/.style={sibling distance=6cm}, level
  2/.style={sibling distance=3cm}, level 3/.style={sibling
    distance=1.5cm}, level 4/.style={sibling distance=1cm}, scale=0.8,emph/.style={edge from parent/.style={black, dashed,draw}}]
\node[draw,circle,label=left:{1}] {$\root$} [grow'=down]
child{node[draw,circle,label=left:{2}]  {}
    child{node[draw,circle,label=left:{1}] {}
      child{node[draw,circle,label=left:{0}] {}
       child[emph]{}
       child[emph]{}
     }
     }
         child{node[draw,circle,label=left:{1}] {}
      child{node[draw,circle,label=left:{0}] {}
       child[emph]{}
       child[emph]{}}
     child{node[draw,circle,label=left:{0}] {}
        child[emph]{}}}  
      child{node[draw,circle,label=left:{1}]{}
        child{node[draw,circle,label=left:{0}]{}
          child[emph]{} }
      }
  } 
child{node[draw,circle,label=left:{2}]  {}
  child{node[draw,circle,label=left:{1}] {}
    child{node[draw,circle,label=left:{0}] {}
      child[emph]{}
      child[emph]{}}}
    child{node[draw,circle,label=left:{1}] {}
    child{node[draw,circle,label=left:{0}] {}
      child[emph]{}}
    child{node[draw,circle,label=left:{0}] {}
      child[emph]{}
      child[emph]{}}}
               }; 
\end{tikzpicture}
\caption{The values of the function $g_1$ are on the left of each vertex.}\label{fig:h1}
\end{figure}
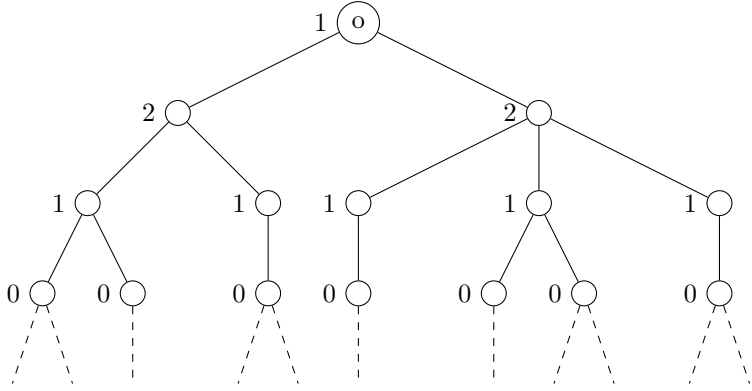

Clearly, $g_1 \in \Lip$ and $\| g_1 \|=1$. But we also have
\[
  (Bg_1)(v)=  \begin{cases}
    2 \gamma(\root), & \text{ if } v = \root, \\
    \gamma(v), & \text{ if } v \in \Chi{\root},\\
    0, & \text{ elsewehere.}
  \end{cases}
\]
Then $\| Bg_1 \| \geq 2 \gamma(\root)$ and thus $\| B \| \geq 2
\gamma(\root)$. This estimate is part of the next proposition.

\begin{proposition}\label{prop:estimate_B_below}
Let $T$ be a rooted, countable infinite and locally finite tree. Assume that $T$
is homogeneous by sectors at the level $N$. Then $\| B \| \geq \max\{2\gamma(\root), 3 \Gamma'-2 \}$, where $\Gamma':=\max\{ \gamma(v) \, : \, 0<|v|\leq N \}$.
\end{proposition}
\begin{proof}
Let $u^* \in T^*$ such that $\gamma(u^*)=\Gamma'$. Define the function
$g_2: T \to \C$ as (see Figure~\ref{fig:h2} for an example)
\[
    g_2(v)=\begin{cases}
      1,& \text{ if } v=\parent{u^*},\\
      2,& \text{ if } v=u^*,\\
      3,& \text{ if } v \in \Chi{u^*}\\
      2,& \text{ if } v \in \nChi{u^*}{2},\\
      1,& \text{ if } v \in \nChi{u^*}{3},\\
      0,& \text{ in any other case.}
    \end{cases}
\]  

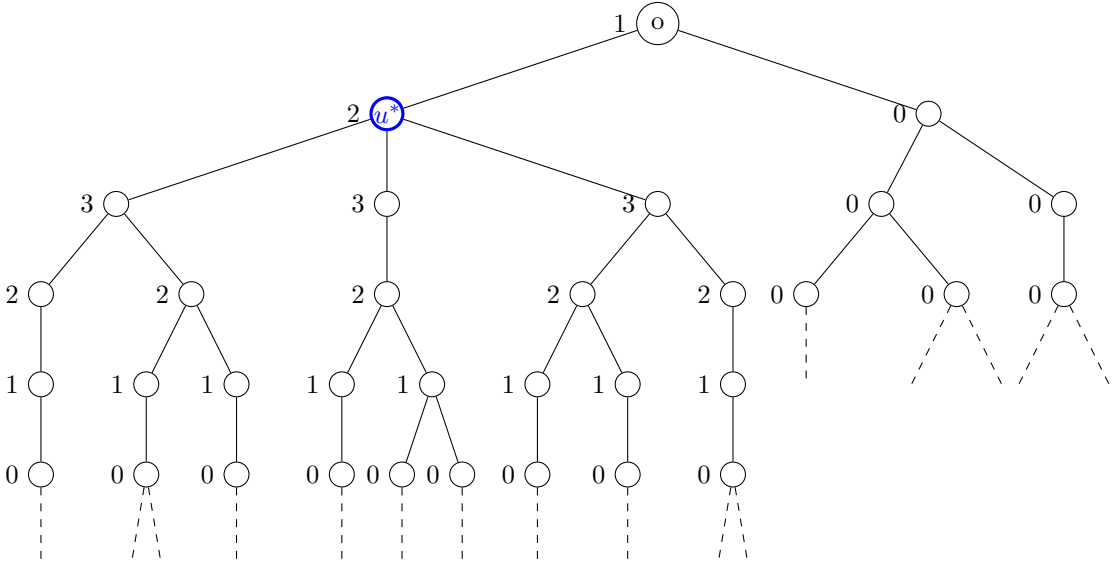
\begin{figure}[!htbp]
\centering
\begin{tikzpicture}[level 1/.style={sibling distance=9cm}, level
  2/.style={sibling distance=4.5cm}, level 3/.style={sibling
    distance=2.5cm}, level 4/.style={sibling distance=1.5cm},level
  5/.style={sibling distance=1cm}, level 6/.style={sibling distance=0.5cm}, scale=0.8,emph/.style={edge from parent/.style={black, dashed,draw}}]
\node[draw,circle,label=left:{1}] {$\root$} [grow'=down]
child{node[draw,circle,label=left:{0}]  {}
    child{node[draw,circle,label=left:{0}] {}
      child{node[draw,circle,label=left:{0}] {}
       child[emph]{}
       child[emph]{}
     }
     }
     child{node[draw,circle,right=1cm,label=left:{0}] {}
      child{node[draw,circle,label=left:{0}] {}
       child[emph]{}
       child[emph]{}}
     child{node[draw,circle,label=left:{0}] {}
        child[emph]{}}}
  } 
child{node[arn_r,draw,circle,label=left:{2}]  {$u^*$}
  child{node[draw,circle,label=left:{3}] {}
    child{node[draw,circle,label=left:{2}] {}
      child{node[draw,circle,label=left:{1}] {}
        child{node[draw,circle,label=left:{0}] {}
         child[emph]{}
         child[emph]{}}}}
  child{node[draw,circle,label=left:{2}] {}
    child{node[draw,circle,label=left:{1}] {}
      child{node[draw,circle,label=left:{0}] {}
        child[emph]{}}}
    child{node[draw,circle,label=left:{1}] {}
      child{node[draw,circle,label=left:{0}] {}
        child[emph]{}}}
    }}
  child{node[draw,circle,label=left:{3}] {}
  child{node[draw,circle,label=left:{2}] {}
    child{node[draw,circle,label=left:{1}] {}
      child{node[draw,circle,label=left:{0}] {}
        child[emph]{}}
    child{node[draw,circle,label=left:{0}] {}
        child[emph]{}}}
    child{node[draw,circle,label=left:{1}] {}
      child{node[draw,circle,label=left:{0}] {}
        child[emph]{}}}
    }}
  child{node[draw,circle,label=left:{3}] {}
    child{node[draw,circle,label=left:{2}] {}
      child{node[draw,circle,label=left:{1}] {}
        child{node[draw,circle,label=left:{0}] {}
          child[emph]{}}}
      child{node[draw,circle,label=left:{1}] {}
        child{node[draw,circle,label=left:{0}] {}
          child[emph]{}
          child[emph]{}}}}
  child{node[draw,circle,label=left:{2}] {}
      child{node[draw,circle,label=left:{1}] {}
      child{node[draw,circle,label=left:{0}] {}
        child[emph]{}}}
    }}
               }; 
\end{tikzpicture}
\caption{The values of the function $g_2$ are on the left of each vertex.}\label{fig:h2}
\end{figure}

Clearly, $g_2 \in \Lip$ and  $\| g_2 \|=1$. It is easy to check that
$(Bg_2)(u^*)=3 \gamma(u^*)$ and $(Bg_2)(\parent{u^*})=2$. Hence
$(Bg_2)'(u^*)=3 \gamma(u^*)-2 = 3 \Gamma'-2$. Therefore, $\| B \| \geq 3
\Gamma'-2$. By the calculations before the statement of this
proposition, it follows that $\| B \| \geq \max\{2\gamma(\root), 3
\Gamma'-2 \}$.
\end{proof}

Observe that the same estimate holds for $B$ as operator on
$\Lip_0$. For a particular type of trees, the above estimate can be
improved.
\begin{proposition}\label{prop:estimate_B_below2}
Let $T$ be a rooted, countable infinite and locally finite tree. Assume that $T$
is homogeneous by sectors at the level $N$. Let $\Gamma'':=\max\{
\gamma(v) \, : \, |v|=1 \}$. 
Then $\| B \| \geq  \Gamma'' + 2 \gamma(\root) -2$.
\end{proposition}
\begin{proof}
Let $u^* \in T$, with $|u^*|=1$, be such that $\gamma(u^*)=\Gamma''$. Define the
function $g_3: T \to \C$ as (see Figure~\ref{fig:h3} for an example)
\[
    g_3(v)=\begin{cases}
      -1,& \text{ if } v=\root,\\
      0,& \text{ if } v=u^*,\\
      1,& \text{ if } v \in \Chi{u^*}\\
      -2,& \text{ if } v \in \Chi{\root}, v\neq u^* \\
      -1,& \text{ if } v \in \nChi{\root}{2}, v \notin \Chi{u^*},\\
      0,& \text{ in any other case.}
    \end{cases}
\]  

\begin{figure}[!htbp]
\centering
\begin{tikzpicture}[level 1/.style={sibling distance=7cm}, level
  2/.style={sibling distance=3cm}, level 3/.style={sibling
    distance=1cm}, level 4/.style={sibling distance=0.5cm},
  scale=0.9,emph/.style={edge from parent/.style={black,
      dashed,draw}}]
  \node[draw,circle,label=left:{-1}] {$\root$} [grow'=down]
  child{node[draw,circle,label=left:{-2}]  {}
    child{node[draw,circle,label=left:{-1}] {}
      child{node[draw,circle,label=left:{0}] {}
        child[emph]{}
        child[emph]{}
      }
    }
    child{node[draw,circle,right=1cm,label=left:{-1}] {}
      child{node[draw,circle,label=left:{0}] {}
        child[emph]{}
        child[emph]{}}
      child{node[draw,circle,label=left:{0}] {}
        child[emph]{}}}
  }
  child{node[arn_r,draw,circle,label=left:{0}]  {$u^*$}
    child{node[draw,circle,label=left:{1}] {}
      child{node[draw,circle,label=left:{0}] {}
        child[emph]{}}}
    child{node[draw,circle,label=left:{1}] {}
      child{node[draw,circle,label=left:{0}] {}
        child[emph]{}
        child[emph]{}}
      child{node[draw,circle,label=left:{0}] {}
        child[emph]{}}}
    child{node[draw,circle,label=left:{1}] {}
      child{node[draw,circle,label=left:{0}] {}
        child[emph]{}
        child[emph]{}}
      child{node[draw,circle,label=left:{0}] {}
        child[emph]{}}
      child{node[draw,circle,label=left:{0}] {}
        child[emph]{}}}
    child{node[draw,circle,label=left:{1}] {}
      child{node[draw,circle,label=left:{0}] {}
        child[emph]{}
        child[emph]{}}
    }
  };
\end{tikzpicture}
\caption{The values of the function $g_3$ are on the left of each vertex.}\label{fig:h3}
\end{figure}
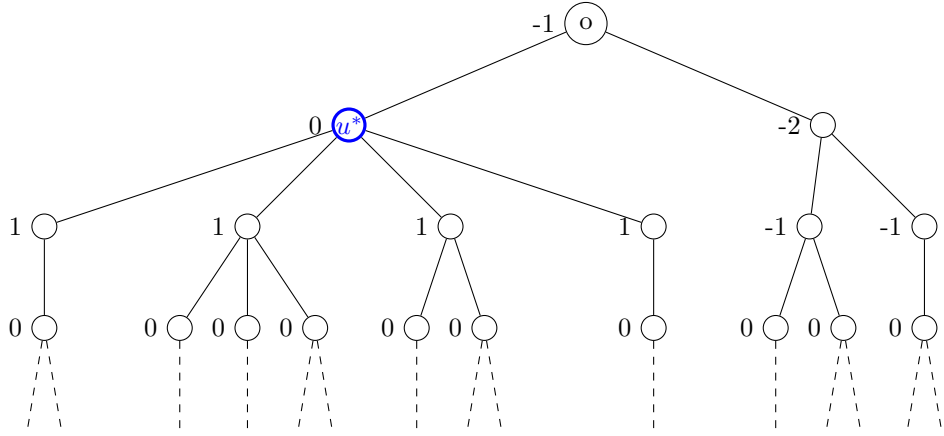

Clearly, $g_3 \in \Lip$ and it is straightforward to check that $\| g_3
\|=1$. But also, $(Bg_3)(u^*)=\gamma(u^*)=\Gamma''$ and $(Bg_3)(\parent{u^*})=-2(\gamma(\root)-1)$. Hence
$(Bg_3)'(u^*)=\Gamma''+2\gamma(\root) -2 $. Therefore, $\| B \| \geq
\Gamma'' + 2 \gamma(\root) -2$, as desired.
\end{proof}

It is easily verified that the estimate in Proposition
\ref{prop:estimate_B_below2} is better than the estimate in
Proposition \ref{prop:estimate_B_below} if $\gamma(\root)>\Gamma''
= \Gamma'>2$.

How good are the estimates in Propositions \ref{prop:boundB} and
\ref{prop:estimate_B_below}? In the next section, we will show that
for a homogeneous tree, we can find the precise value of $\|B\|$. For
now, we show what happens for the case of a tree that is homogeneous
by sectors at the level $N=1$.

\begin{proposition}
Let $T$ be a rooted, countable infinite and locally finite tree. Assume that $T$
is homogeneous by sectors at the level $1$. Then
\[
  \| B \| = \max\{2\gamma(\root), \Lambda \},
\]
both as an operator on $\Lip$ and on $\Lip_0$.
\end{proposition}
\begin{proof}
Let $n=\gamma(\root)$ and let $v_1, v_2, \dots, v_n$ be the children of
$\root$. Furthermore, assume that $\gamma(v_1) \leq \gamma(v_2) \leq \dots
\leq \gamma(v_n)$.

Let us deal with the case $\gamma(v_n)=1$ first. It is easy to check that $\Lambda=2n$ if there exists $j$ with $\gamma(v_j)=0$, and $\Lambda=2n-1$ if $\gamma(v_j)=1$ for all $j=1, 2, \dots, n$. Hence, by Proposition \ref{prop:boundB}, $\|B\|\leq 2n=2\gamma(\root)$. By proposition \ref{prop:estimate_B_below}, it follows that  $\|B\|=2\gamma(\root)$ and thus $\|B\|=\max\{2\gamma(\root),\Lambda\}$, as desired.

So assume for the rest of the proof that $\gamma(v_n)\geq 2$. We have two cases:
\begin{itemize}
\item Assume $\gamma(v_n) \geq n$. We have three cases:
\begin{itemize}
\item[$\ast$] If $\gamma(v_j)=0$ for some $j=1, 2 \dots, n-1$, we have that
\begin{align*}
    \gamma(v_j)+\gamma(\parent{v_j}) -1 + |\gamma(v_j)-1|+
  |\gamma(v_j)-\gamma(\parent{v_j})| \, |v_j| &= 2n \\
                                              &\leq 2 \gamma(v_n) \\
                                              &\leq 3\gamma(v_n)-2.
\end{align*}
  
\item[$\ast$] If $1\leq  \gamma(v_j) \leq n$ for some $j=1, 2 \dots, n-1$, we have that
\begin{align*}
  \gamma(v_j)+\gamma(\parent{v_j}) -1 + |\gamma(v_j)-1|+ |\gamma(v_j)-\gamma(\parent{v_j})| \, |v_j|
  &= 2 \gamma(v_j) +n -2 + |\gamma(v_j)-n| \\
   &= \gamma(v_j) +2 n -2 \\
  &\leq 3 \gamma(v_n) -2.
\end{align*}

\item[$\ast$] If $n\leq  \gamma(v_j)$ for some $j=1, 2 \dots, n$, we have that
\begin{align*}
  \gamma(v_j)+\gamma(\parent{v_j}) -1 + |\gamma(v_j)-1|+ |\gamma(v_j)-\gamma(\parent{v_j})| \, |v_j|
  &= 2 \gamma(v_j) +n -2 + |\gamma(v_j)-n| \\
   &= 3 \gamma(v_j) -2 \\
  &\leq 3 \gamma(v_n) -2. 
\end{align*}
\end{itemize}

Thus, $\Lambda= 3 \gamma(v_n) -2$, and by Proposition
\ref{prop:boundB} $\|B\|\leq \max\{2n, 3 \gamma(v_n)-2 \}$. But $2n
\leq 2 \gamma(v_n) \leq 3 \gamma(v_n) -2$, since $\gamma(v_n)\geq
2$. Therefore $\|B \|\leq \Lambda= 3\gamma(v_n) -2$. But by
Proposition \ref{prop:estimate_B_below} we also have $\|B \| \geq 3\gamma(v_n)
-2$, thus  $\|B \| = 3\gamma(v_n)-2$. Therefore $\|B \|= \max\{2\gamma(\root), \Lambda\}$, as desired.

\item Assume $n > \gamma(v_n)$.
\begin{itemize}
\item[$\ast$] If $\gamma(v_j)=0$ for some $j=1, 2, \dots, n-1$, then
\begin{align*}
    \gamma(v_j) +\gamma(\parent{v_j}) -1 + |\gamma(v_j)-1| +
  |\gamma(v_j)-\gamma(\parent{v_j})| \, |v_j| &= 2n \\
  &\leq \gamma(v_n)+2n-2.
\end{align*}
\item[$\ast$] If $1\leq  \gamma(v_j)$ for some $j=1, 2 \dots, n$, then
\begin{align*}
  \gamma(v_j)+\gamma(\parent{v_j}) -1 + |\gamma(v_j)-1|+ |\gamma(v_j)-\gamma(\parent{v_j})| \, |v_j|
  &= 2 \gamma(v_j) +n -2 + |\gamma(v_j)-n| \\
  &= \gamma(v_j) +2 n -2 \\
  &\leq \gamma(v_n) +2 n -2.
\end{align*}
\end{itemize}

Hence, $\Lambda=\gamma(v_n) +2n-2$, and by Proposition
\ref{prop:boundB}, we have $\|B \|\leq \max\{2n, \gamma(v_n)
+2n-2\}=\gamma(v_n) +2n-2$. But now, by Proposition
\ref{prop:estimate_B_below2}, we have that  $\|B \| \geq \gamma(v_n)
+2n-2$, and therefore  $\|B \|= \gamma(v_n) +2n-2$. Thus $\|B \|= \max\{2\gamma(\root), \Lambda\}$, as desired.\qedhere
\end{itemize}
\end{proof}

Are the estimates in Propositions \ref{prop:boundB},
\ref{prop:estimate_B_below} and \ref{prop:estimate_B_below2} sharp for
$N\geq 2$? We have not been able to obtain an answer and we leave open
the question of what the norm of $B$ is for general trees.

{\bf Note: }{\em We would like to thank a referee for suggesting we
  investigate lower estimates for the norm of $B$, which led to
  improvements of the estimates we had in a previous version of this paper.}

\section{Norm of $B$ and of $B^n$ on Homogeneous Trees}\label{sec_norm}

In this section, we find an expression for the norms of $B$ and $B^n$ for the case when $T$ is a homogeneous tree. Recall that a tree is homogeneous of
order $\gamma$ if $\gamma(v)=\gamma$ for all $v \in T$.

\begin{theorem}
Let $T$ be a rooted homogeneous tree of order $\gamma$ and let $B$ be the
backward shift on $\Lip$. Then
\[
  \| B \| = \begin{cases}
    2,& \text{ if } \gamma=1, \\
    3 \gamma -2, & \text{ if } \gamma \geq 2.
  \end{cases}
\]
\end{theorem}
\begin{proof}
First of all, observe that, by Corollary \ref{cor:boundB}, since
$\Gamma=\gamma$ and $\Omega=0$ we have $\| B \| \leq \max\{ 2\gamma, 3
\gamma-2 \}$.

By Proposition \ref{prop:estimate_B_below}, we have that $\| B \| \geq \max\{ 2\gamma, 3 \gamma-2 \}$, and hence the result follows.
\end{proof}

We have the same result for the little Lipschitz space.
\begin{theorem}
Let $T$ be a rooted homogeneous tree of order $\gamma$ and let $B$ be the
backward shift on $\Lip_0$. Then
\[
  \| B \| = \begin{cases}
    2,& \text{ if } \gamma=1, \\
    3 \gamma -2, & \text{ if } \gamma \geq 2.
  \end{cases}
\]
\end{theorem}
For the computation of the spectrum of $B$ for a rooted homogeneous tree, we will need to find the
norm of $B^n$. Let us do some preliminary computations. First of all,
it is clear that, for any $f \in \calF$ we have
\[
  (B^n f)(v)=\sum_{w \in \nChi{v}{n}} f(w).
\]
From this, and since each vertex in $ \nChi{v}{n-1}$ is the parent of $\gamma$
vertices in $\nChi{v}{n}$, it follows that 
\begin{equation}\label{eq:Bnf}
  \begin{split}
    (B^n f)(v)
    &= \sum_{w \in \nChi{v}{n}} (f(w) - f(\parent{w})) + \sum_{w \in
      \nChi{v}{n}} f(\parent{w}) \\
    &= \sum_{w \in \nChi{v}{n}} f'(w) + \gamma \sum_{w
      \in \nChi{v}{n-1}} f(w).
  \end{split}
\end{equation}
In the same manner, we have
\[
  (B^n f)(v)= \sum_{w \in \nChi{v}{n}} f'(w) + \gamma \left( \sum_{w
      \in \nChi{v}{n-1}} (f(w)-f(\parent{w})) +  \gamma \sum_{w \in
      \nChi{v}{n-2}} f(w)\right),
\]
since each vertex in $ \nChi{v}{n-2}$ is the parent of $\gamma$
vertices in $\nChi{v}{n-1}$. Proceeding inductively, we get
\begin{equation*}
  \begin{split}
    (B^n f)(v)
    &= \sum_{w \in \nChi{v}{n}} f'(w) + \gamma \sum_{w \in
      \nChi{v}{n-1}} f'(w) +  \gamma^2 \sum_{w \in \nChi{v}{n-2}} f'(w) \\
    & \hspace{0.4cm} + \gamma^3 \sum_{w \in \nChi{v}{n-3}} f'(w) + \dots +
    \gamma^{n-2} \sum_{w \in \nChi{v}{2}} f'(w) + \gamma^{n-1} \sum_{w
      \in \Chi{v}} f'(w) + \gamma^n f(v).
  \end{split}
\end{equation*}
In short, we have obtained
\begin{equation}\label{eq:Bn_long}
\begin{split}
  (B^n f)(v)
    &=\sum_{\nChi{u}{n}} f(w) \\
    &= \sum_{k=0}^{n-1} \gamma^{k} \left( \sum_{w \in \nChi{v}{n-k}}
      f'(w) \right) + \gamma^n f(v).
  \end{split}
\end{equation}

We will use this expression in the proof of the following proposition.
\begin{proposition}\label{prop:Bn}
Let $T$ be a rooted homogeneous tree of order $\gamma$ and let $B$ be the
backward shift on $\Lip$. Then,
\[
  \| B^n \| \leq \max\{ (2n+1)\gamma^n-2n\gamma^{n-1},(n+1)\gamma^n \}
  \]
\end{proposition}
\begin{proof}
  Using Equation \eqref{eq:Bn_long}, we have
  \[
    \mod{ (B^n f)(\root) } \leq \sum_{k=0}^{n-1} \left( \gamma^k \sum_{w \in \nChi{\root}{n-k}}
    |f'(w)| \right) + \gamma^n |f(\root)|.
  \]
  Since for every $s \in \N$ there are $\gamma^s$ vertices in $\nChi{\root}{s}$, and
  $|f'(w)|\leq \| f \|$ for all $w \in T$, we get
  \[
    \mod{ (B^n f)(\root)}  \leq \sum_{k=0}^{n-1} \gamma^k
    \gamma^{n-k} \| f \| + \gamma^n |f(\root)|,
  \]
  and therefore, since $|f(\root)| \leq \| f \|$, we have
  \[
    \mod{ (B^n f)(\root) } \leq (n+1) \gamma^n \| f \|.
  \]

  Now, let $v \in T^*$.
  Equation \eqref{eq:Bnf} is
\[
(B^n f)(v)= \sum_{w \in \nChi{v}{n}} f'(w) + \gamma \sum_{w \in
  \nChi{v}{n-1}} f(w).
\]

Also, we have 
\[
(B^n f)(\parent{v})= \sum_{w \in \nChi{\parent{v}}{n}} f(w) =
\sum_{\substack{w \in \nChi{\parent{v}}{n} \\ w \notin \nChi{v}{n-1}}}  f(w) 
 + \sum_{w \in
  \nChi{v}{n-1}} f(w).
\]

The two equations above give
\begin{equation}\label{eq:Bnprime}
  \begin{split}
    (B^n f)'(v)
    &= (B^n f)(v)-(B^n f)(\parent{v}) \\
    &=  \sum_{w \in \nChi{v}{n}} f'(w) + (\gamma-1) \sum_{w \in
      \nChi{v}{n-1}} f(w) - \sum_{\substack{w \in \nChi{\parent{v}}{n}
        \\ w \notin \nChi{v}{n-1}}}  f(w).
\end{split}
\end{equation}

  Since each vertex in $\nChi{\parent{v}}{n-1}$ is the parent
  of $\gamma$ vertices in $\nChi{\parent{v}}{n}$, we have
  \begin{equation*}
    \begin{split}
    \sum_{\substack{w \in \nChi{\parent{v}}{n} \\ w \notin
        \nChi{v}{n-1}}} f(w)
    & = \sum_{\substack{w \in \nChi{\parent{v}}{n} \\ w \notin
            \nChi{v}{n-1}}} \left( f(w) -f(\parent{w}) \right) +
        \sum_{\substack{w \in \nChi{\parent{v}}{n} \\ w \notin
            \nChi{v}{n-1}}} f(\parent{w}) \\
    &= \sum_{\substack{w \in \nChi{\parent{v}}{n} \\ w \notin
        \nChi{v}{n-1}}} f'(w) + \gamma \sum_{\substack{w \in
        \nChi{\parent{v}}{n-1} \\ w \notin \nChi{v}{n-2}}} f(w).
     \end{split}
  \end{equation*}
            
Inductively, we obtain
\begin{equation*}
  \begin{split}
    \sum_{\substack{w \in \nChi{\parent{v}}{n} \\ w \notin
        \nChi{v}{n-1}}} f(w)
    =& \sum_{\substack{w \in \nChi{\parent{v}}{n} \\ w \notin
            \nChi{v}{n-1}}} f'(w)
    + \gamma \sum_{\substack{w \in \nChi{\parent{v}}{n-1} \\ w \notin
        \nChi{v}{n-2}}} f'(w) 
    + \gamma^2 \sum_{\substack{w \in \nChi{\parent{v}}{n-2} \\ w \notin
        \nChi{v}{n-3}}} f'(w) \\
    & \qquad \qquad \quad + \dots
    + \gamma^{n-2} \sum_{\substack{w \in \nChi{\parent{v}}{2} \\ w \notin
        \Chi{\parent{v}}}} f'(w)
    + \gamma^{n-1} \sum_{\substack{w \in \Chi{\parent{v}} \\ w \neq v}} f(w).
    \end{split}
  \end{equation*}
Since there are  $\gamma-1$ vertices in $\Chi{\parent{v}}$ different
than $v$, we have
\begin{equation*}
   \begin{split}
    \sum_{\substack{w \in \nChi{\parent{v}}{n} \\ w \notin
        \nChi{v}{n-1}}} f(w)
    =& \sum_{\substack{w \in \nChi{\parent{v}}{n} \\ w \notin
            \nChi{v}{n-1}}} f'(w)
    + \gamma \sum_{\substack{w \in \nChi{\parent{v}}{n-1} \\ w \notin
        \nChi{v}{n-2}}} f'(w) 
    + \gamma^2 \sum_{\substack{w \in \nChi{\parent{v}}{n-2} \\ w \notin
        \nChi{v}{n-3}}} f'(w) \\
    & \quad  + \dots
    + \gamma^{n-2} \sum_{\substack{w \in \nChi{\parent{v}}{2} \\ w \notin
        \Chi{\parent{v}}}} f'(w)
    + \gamma^{n-1} \sum_{\substack{w \in \Chi{\parent{v}} \\ w \neq
        v}} f'(w) + \gamma^{n-1} (\gamma-1) f(\parent{v}).
    \end{split}
  \end{equation*}
In short, we have obtained
  \begin{equation}\label{eq:chi_n_n-1}
    \sum_{\substack{w \in \nChi{\parent{v}}{n} \\ w \notin
        \nChi{v}{n-1}}} f(w)
    = \sum_{k=0}^{n-1} \gamma^{k}\left( \sum_{\substack{w\in
          \nChi{\parent{v}}{n-k} \\ w \notin \nChi{v}{n-k-1}}} f'(w)
    \right) +  \gamma^{n-1} (\gamma-1) f(\parent{v}).
  \end{equation}
  
Substituting Equation \eqref{eq:Bn_long} (for $B^{n-1}$) and \eqref{eq:chi_n_n-1} into Equation \eqref{eq:Bnprime} we obtain
\begin{equation*}
  \begin{split}
    (B^n f)'(v)
    & = (B^n f)(v)-(B^n f)(\parent{v}) \\
    &= \sum_{w \in \nChi{v}{n}} f'(w) + (\gamma-1) \left(
      \sum_{k=0}^{n-2} \gamma^{k} \left( \sum_{w \in \nChi{v}{n-k-1}}
        f'(w) \right)+ \gamma^{n-1} f(v) \right) \\
    & \qquad -  \sum_{k=0}^{n-1} \gamma^{k}\left( \sum_{\substack{w\in
          \nChi{\parent{v}}{n-k} \\ w \notin \nChi{v}{n-k-1}}} f'(w)
    \right) -  \gamma^{n-1} (\gamma-1) f(\parent{v}) \\
    &= \sum_{w \in \nChi{v}{n}} f'(w) + (\gamma-1) \left(
      \sum_{k=0}^{n-2} \gamma^{k} \left( \sum_{w \in \nChi{v}{n-k-1}}
        f'(w) \right) \right) \\
    & \qquad -  \sum_{k=0}^{n-1} \gamma^{k}\left( \sum_{\substack{w\in
          \nChi{\parent{v}}{n-k} \\ w \notin \nChi{v}{n-k-1}}} f'(w)
    \right) +  \gamma^{n-1} (\gamma-1) f'(v)
  \end{split}
\end{equation*}
Since, for every $s \in \N$, there are $\gamma^s$ vertices in
$\nChi{v}{s}$, and for every $k \in \{1, 2, \dots, n-1\}$ there are
$\gamma^{n-k} - \gamma^{n-k-1}$ vertices in $\nChi{\parent{v}}{n-k}
\setminus \nChi{v}{n-k-1}$, we obtain
\begin{equation*}
  \begin{split}
    \mod{ ( B^n f)'(v)) }
    & \leq \gamma^n \| f \| + (\gamma-1) \sum_{k=0}^{n-2}
     \gamma^k  \gamma^{n-k-1}\| f \| + \sum_{k=0}^{n-1} \gamma^{k}
    (\gamma^{n-k} - \gamma^{n-k-1}) \| f \| +  \gamma^{n-1} (\gamma-1)
    \| f \| \\
    & = \left( \gamma^n + (\gamma-1) \sum_{k=0}^{n-2}
    \gamma^{n-1} + \sum_{k=0}^{n-1}
    (\gamma^{n} - \gamma^{n-1})  +  \gamma^{n-1} (\gamma-1) \right)
  \| f \| \\
    & = \left( \gamma^n + (n-1) (\gamma-1) \gamma^{n-1} + n
      (\gamma^{n} - \gamma^{n-1})  +  \gamma^{n-1} (\gamma-1) \right)
    \| f \| \\
    & = \left( (2n+1) \gamma^n - 2n \gamma^{n-1} \right) \| f \|.
    \end{split}
 \end{equation*}
 Therefore, $ \| B^n \| \leq \max\{(2n+1) \gamma^n - 2n
 \gamma^{n-1},(n+1) \gamma^n\}$,  as desired.
\end{proof}

Using this proposition we can compute the exact value of the norm of $B^n$.

\begin{theorem}\label{th:Bn}
Let $T$ be a rooted homogeneous tree of order $\gamma$ and let $B$ be the
backward shift on $\Lip$. Then, for every $n \in \N$,
\[
  \| B^n \| = \begin{cases}
    n+1,& \text{ if } \gamma=1, \\
    (2n+1) \gamma^n - 2n \gamma^{n-1}, & \text{ if } \gamma \geq 2.
  \end{cases}
\]
\end{theorem}
\begin{proof}
First of all, observe that, by Proposition \ref{prop:Bn} we have $\| B
\| \leq \max\{ (2n+1) \gamma^n - 2n \gamma^{n-1},n+1 \}$, which equals
$n+1$ if $\gamma=1$, and $(2n+1) \gamma^n - 2n \gamma^{n-1}$ if $\gamma
\geq 2$.

Choose a fixed $u^* \in \Chi{\root}$. Define the function $h_{n}: T \to \C$ (see Figure~\ref{fig:h} for an example) as
\[
    h_{n}(v)=\begin{cases}
      1,& \text{ if } v=\root,\\
      |v|+1,& \text{ if } v \in S_{u^*} \text{ and } 1 \leq |v| \leq n+1,\\
      (2n+3)-|v|,& \text{ if } v \in S_{u^*} \text{ and } n+1 \leq |v| \leq 2n+3,\\
      -|v|+1,& \text{ if } v \notin S_{u^*} \text{ and } 1 \leq |v|
      \leq n,\\
      -(2n-1)+|v|,& \text{ if }  v \notin S_{u^*} \text{ and } n \leq |v|
      \leq 2n-1,\\
      0,& \text{ in any other case.}
    \end{cases}
\]

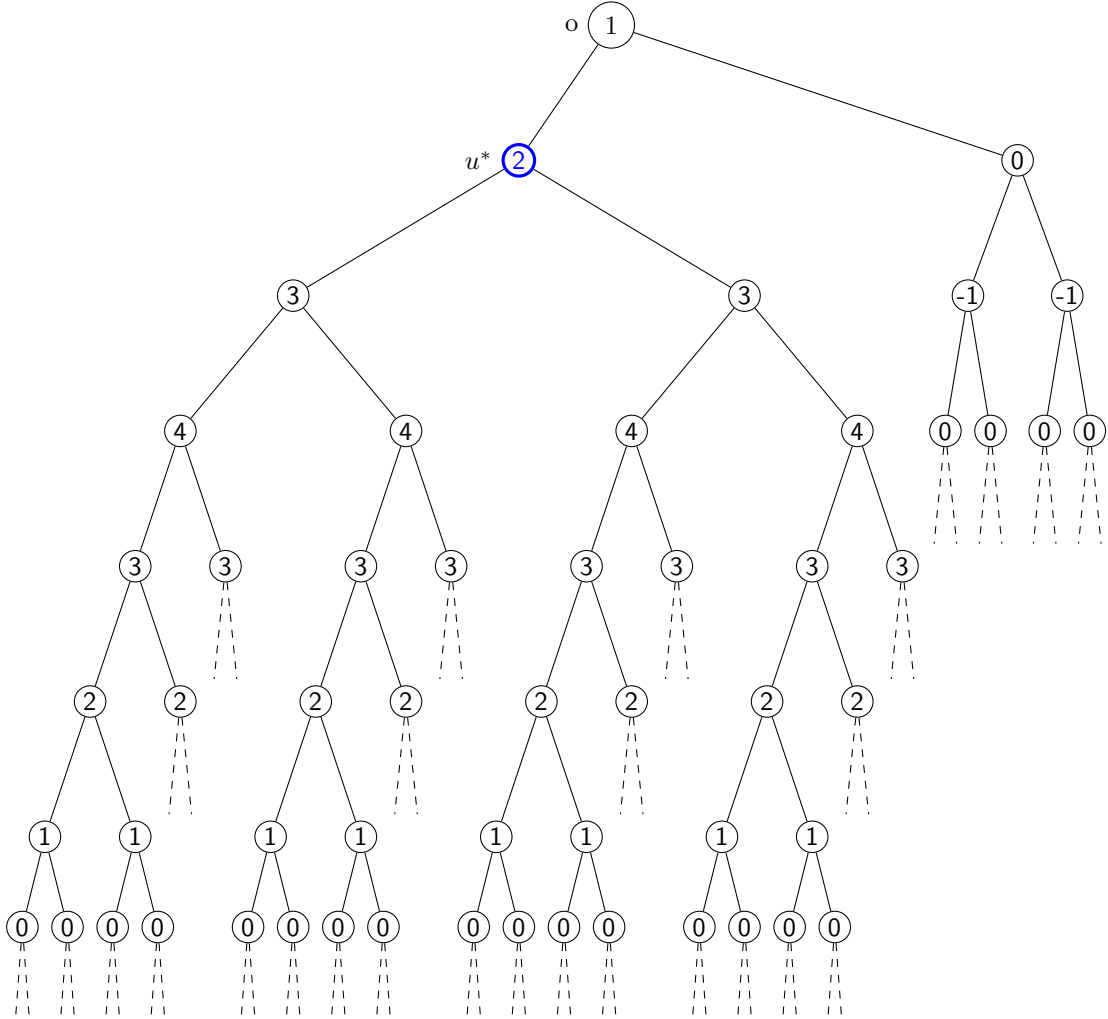
\begin{figure}[!htbp]
\begin{tikzpicture}[%
  level 1/.style={sibling distance=32cm, level distance=3cm},%
  level 2/.style={sibling distance=17cm, level distance=3cm},%
  level 3/.style={sibling distance=8.5cm, level distance=3cm},%
  level 4/.style={sibling distance=5cm, level distance=3cm},%
  level 5/.style={sibling distance=4.2cm, level distance=3cm},%
  level 6/.style={sibling distance=2cm, level distance=3cm},%
  level 7/.style={sibling distance=1cm,level distance=2cm},%
  level 8/.style={sibling distance=0.3cm, level distance=2cm},%
  scale=0.6,emph/.style={edge from parent/.style={black, dashed,draw}}]
  \node[draw,circle,label=left:{$\root$}] {1} [grow=down]
child[sibling distance=0cm]{node[left=1cm,arn_r,draw,circle,label=left:{$u^*$}] {2}
  child[sibling distance=10cm]{node[arn_n,draw,circle] {3}
    child[sibling distance=5cm]{node[arn_n,draw,circle] {4}
      child[sibling distance=2cm]{node[arn_n,draw,circle] {3}
        child[sibling distance=2cm]{node[arn_n,draw,circle] {2}
          child{node[arn_n,draw,circle] {1}
            child{node[arn_n,draw,circle] {0}
              child[emph]{}
              child[emph]{}
            }
            child{node[arn_n,draw,circle] {0}
              child[emph]{}
              child[emph]{}
            }
          }
          child{node[arn_n,draw,circle] {1}
            child{node[arn_n,draw,circle] {0}
              child[emph]{}
              child[emph]{}
            }
            child{node[arn_n,draw,circle] {0}
              child[emph]{}
              child[emph]{}
            }
          }
        }
        child[sibling distance=2cm]{node[arn_n,draw,circle] {2}
          child[sibling distance=0.5cm,level distance=2.5cm,emph]{}
          child[sibling distance=0.5cm,level distance=2.5cm,emph]{}
        }
      }
      child[sibling distance=2cm]{node[arn_n,draw,circle] {3}
        child[sibling distance=0.5cm,level distance=2.5cm,emph]{}
        child[sibling distance=0.5cm,level distance=2.5cm,emph]{}
      }
    }
    child[sibling distance=5cm]{node[arn_n,draw,circle] {4}
      child[sibling distance=2cm]{node[arn_n,draw,circle] {3}
        child[sibling distance=2cm]{node[arn_n,draw,circle] {2}
          child{node[arn_n,draw,circle] {1}
            child{node[arn_n,draw,circle] {0}
              child[emph]{}
              child[emph]{}
            }
            child{node[arn_n,draw,circle] {0}
              child[emph]{}
              child[emph]{}
            }
          }
          child{node[arn_n,draw,circle] {1}
            child{node[arn_n,draw,circle] {0}
              child[emph]{}
              child[emph]{}
            }
            child{node[arn_n,draw,circle] {0}
              child[emph]{}
              child[emph]{}
            }
          }
        }
        child[sibling distance=2cm]{node[arn_n,draw,circle] {2}
          child[sibling distance=0.5cm,level distance=2.5cm,emph]{}
          child[sibling distance=0.5cm,level distance=2.5cm,emph]{}
        }
      }
      child[sibling distance=2cm]{node[arn_n,draw,circle] {3}
        child[sibling distance=0.5cm,level distance=2.5cm,emph]{}
        child[sibling distance=0.5cm,level distance=2.5cm,emph]{}
      }
    }
  }
  child[sibling distance=10cm]{node[arn_n,draw,circle] {3}
    child[sibling distance=5cm]{node[arn_n,draw,circle] {4}
      child[sibling distance=2cm]{node[arn_n,draw,circle] {3}
        child[sibling distance=2cm]{node[arn_n,draw,circle] {2}
          child{node[arn_n,draw,circle] {1}
            child{node[arn_n,draw,circle] {0}
              child[emph]{}
              child[emph]{}
            }
            child{node[arn_n,draw,circle] {0}
              child[emph]{}
              child[emph]{}
            }
          }
          child{node[arn_n,draw,circle] {1}
            child{node[arn_n,draw,circle] {0}
              child[emph]{}
              child[emph]{}
            }
            child{node[arn_n,draw,circle] {0}
              child[emph]{}
              child[emph]{}
            }
          }
        }
        child[sibling distance=2cm]{node[arn_n,draw,circle] {2}
          child[sibling distance=0.5cm,level distance=2.5cm,emph]{}
          child[sibling distance=0.5cm,level distance=2.5cm,emph]{}
        }
      }
      child[sibling distance=2cm]{node[arn_n,draw,circle] {3}
        child[sibling distance=0.5cm,level distance=2.5cm,emph]{}
        child[sibling distance=0.5cm,level distance=2.5cm,emph]{}
      }
    }
    child[sibling distance=5cm]{node[arn_n,draw,circle] {4}
      child[sibling distance=2cm]{node[arn_n,draw,circle] {3}
        child[sibling distance=2cm]{node[arn_n,draw,circle] {2}
          child{node[arn_n,draw,circle] {1}
            child{node[arn_n,draw,circle] {0}
              child[emph]{}
              child[emph]{}
            }
            child{node[arn_n,draw,circle] {0}
              child[emph]{}
              child[emph]{}
            }
          }
          child{node[arn_n,draw,circle] {1}
            child{node[arn_n,draw,circle] {0}
              child[emph]{}
              child[emph]{}
            }
            child{node[arn_n,draw,circle] {0}
              child[emph]{}
              child[emph]{}
            }
          }
        }
        child[sibling distance=2cm]{node[arn_n,draw,circle] {2}
          child[sibling distance=0.5cm,level distance=2.5cm,emph]{}
          child[sibling distance=0.5cm,level distance=2.5cm,emph]{}
        }
      }
      child[sibling distance=2cm]{node[arn_n,draw,circle] {3}
        child[sibling distance=0.5cm,level distance=2.5cm,emph]{}
        child[sibling distance=0.5cm,level distance=2.5cm,emph]{}
      }
    }
  }
}
child[sibling distance=18cm]{node[arn_n,draw,circle] {0}
  child[sibling distance=2.2cm]{node[arn_n,draw,circle] {-1}
    child[sibling distance=1cm]{node[arn_n,draw,circle] {0}
      child[sibling distance=0.5cm,level distance=2.5cm,emph]{}
      child[sibling distance=0.5cm,level distance=2.5cm,emph]{}
    }
    child[sibling distance=1cm]{node[arn_n,draw,circle] {0}
      child[sibling distance=0.5cm,level distance=2.5cm,emph]{}
      child[sibling distance=0.5cm,level distance=2.5cm,emph]{}
    }
  }
  child[sibling distance=2.2cm]{node[arn_n,draw,circle] {-1}
    child[sibling distance=1cm]{node[arn_n,draw,circle] {0}
      child[sibling distance=0.5cm,level distance=2.5cm,emph]{}
      child[sibling distance=0.5cm,level distance=2.5cm,emph]{}
    }
    child[sibling distance=1cm]{node[arn_n,draw,circle] {0}
      child[sibling distance=0.5cm,level distance=2.5cm,emph]{}
      child[sibling distance=0.5cm,level distance=2.5cm,emph]{}
    }
  }
};
\end{tikzpicture}
\caption{Some of the values of $h_n$ for the case $\gamma=2$ and $n=2$. The values  are {\em inside} each vertex, the blue vertex is $u^*$ and the root is labeled $\root$.}\label{fig:h}
\end{figure}

It is clear that
\[
  h_{n}'(v)=\begin{cases}
    1, & \text{ if } v=\root,\\
    1, & \text{ if } v\in S_{u^*} \text{ and } 1\leq |v| \leq n+1\\
    -1, & \text{ if } v\in S_{u^*} \text{ and } n+2\leq |v| \leq 2n+3\\
    -1, & \text{ if } v\notin S_{u^*} \text{ and } 1\leq |v| \leq n\\
    1, & \text{ if } v\notin S_{u^*} \text{ and } n+1\leq |v| \leq 2n-1\\
    0,& \text{ in any other case.}
  \end{cases}
\]
and therefore $h_{n} \in \Lip$ and $\| h_{n} \|=1$. Also, a straightforward computation shows that
\[
    (B^n h_{n})(v)=\begin{cases}
      \gamma^{n-1}(n+1)+ (\gamma^n - \gamma^{n-1})(-n+1),& \text{ if } v=\root,\\
      \gamma^n (n+3-|v|),& \text{ if } v \in S_{u^*} \text{ and } 1 \leq |v| \leq n+3,\\
      \gamma^n (-(n-1)+|v|),& \text{ if } v \notin S_{u^*} \text{ and } 1 \leq |v|
      \leq n-1,\\
      0,& \text{ in any other case.}
    \end{cases}
\] 
Hence, 
\begin{equation}\label{eq:Bnhprime}
    (B^n h_{n})'(v)=\begin{cases}
      \gamma^{n-1}(n+1)+ (\gamma^n - \gamma^{n-1})(-n+1),& \text{ if } v=\root,\\
      \gamma^n (n+2) - \left(\gamma^{n-1}(n+1)+ (\gamma^n -
        \gamma^{n-1})(-n+1)\right),& \text{ if } v =u^*, \\
      \gamma^n (-n+2)- \left(\gamma^{n-1}(n+1)+ (\gamma^n -
        \gamma^{n-1})(-n+1)\right),& \text{ if } v \neq u^* \text{ and
      } |v|=1,\\
      -\gamma^n, &\text{ if } v\in S_{u^*} \text{ and } 2 \leq |v|
      \leq n+3,\\ 
      \gamma^n, &\text{ if } v\notin S_{u^*} \text{ and } 2 \leq |v|
      \leq n-1,\\ 
      0,& \text{ in any other case.}
    \end{cases}
  \end{equation}

If $\gamma=1$, Equation \eqref{eq:Bnhprime} simplifies to
\[
    (B^n h_{n})'(v)=\begin{cases}
      n+1,& \text{ if } v=\root,\\
      1,& \text{ if } v =u^* \\
      -1, &\text{ if } 2 \leq |v| \leq n+3\\ 
      0,& \text{ in any other case.}
    \end{cases}
\]  
Hence, if $\gamma=1$, we have then that $\| B^n h_{n} \|=n+1$, which together with
Proposition \ref{prop:Bn} gives that $\| B^n \|=n+1$, as desired.

If $\gamma \geq 2$, Equation \eqref{eq:Bnhprime} simplifies to
\[
      (B^n h_{n})'(v)=\begin{cases}
      \gamma^{n-1}(2n)- \gamma^n(n-1),& \text{ if } v=\root,\\
      \gamma^n (2n+1) - \gamma^{n-1}(2n),& \text{ if } v =u^*, \\
      \gamma^n - \gamma^{n-1}(2n),& \text{ if } v \neq u^* \text{ and
      } |v|=1,\\
      -\gamma^n, &\text{ if } v\in S_{u^*} \text{ and } 2 \leq |v|
      \leq n+3,\\ 
      \gamma^n, &\text{ if } v\notin S_{u^*} \text{ and } 2 \leq |v|
      \leq n-1,\\ 
      0,& \text{ in any other case.}
    \end{cases}
  \]
  It can be checked that
  \[
    \max\{ |\gamma^{n-1}(2n)- \gamma^n(n-1)|, |\gamma^n (2n+1) -
    \gamma^{n-1}(2n)|, | \gamma^n - \gamma^{n-1}(2n)|, |\gamma^n| \}=
    \gamma^n (2n+1) - \gamma^{n-1}(2n)
  \]
  and hence $\| B^n h_{n} \|= \gamma^n (2n+1) -
  \gamma^{n-1}(2n)$, which  together with
  Proposition \ref{prop:Bn} gives
  \[
    \| B^n \|=\gamma^n (2n+1) - \gamma^{n-1}(2n),
  \]
  as desired.
\end{proof}

Observe that, in the previous proof, the function $h_{n}$ is
also in $\Lip_0$. Hence we also obtain
\begin{theorem}\label{th:Bn0}
Let $T$ be a rooted homogeneous tree of order $\gamma$ and let $B$ be the
backward shift on $\Lip_0$. Then
\[
  \| B^n \| = \begin{cases}
    n+1,& \text{ if } \gamma=1, \\
    (2n+1)\gamma^n - (2n) \gamma^{n-1}, & \text{ if } \gamma \geq 2.
  \end{cases}
\]
\end{theorem}

\section{Spectrum of $B$ on Homogeneous Trees}\label{sec_spectrum}

In this section, we compute the spectrum of $B$ for both the Lipschitz
and the little Lipschitz space in the case where $T$ is a rooted homogeneous
tree. First, we obtain part of the set of eigenvalues. We will show
later that we actually have an equality.

\begin{theorem}\label{th:sigmapB}
Let $T$ be a rooted homogeneous tree of order $\gamma$. If $B$ is the
backward shift on $\Lip$, then
\[
\{ \lambda \in \C \, : \, |\lambda| \leq \gamma \}\subseteq \sigma_{\operatorname{p}}(B).
\]
If $B$ is the backward shift on $\Lip_0$, then
\[
\{ \lambda \in \C \, : \, |\lambda|< \gamma \} \cup \{ \gamma \} \subseteq \sigma_{\operatorname{p}}(B).
\]
\end{theorem}
\begin{proof}
Define $f_\lambda \in \calF$ as $f_\lambda(v)=
(\lambda/\gamma)^{|v|}$. Then,
\[
(Bf_\lambda)(v)=\sum_{w \in \Chi{v}} f_\lambda(w)= \sum_{w \in
  \Chi{v}} (\lambda/\gamma)^{|w|} = \gamma (\lambda/\gamma)^{|v|+1} =
(\lambda f_\lambda)(v).
\]
For $v \in T^*$ we have
\[
  f_\lambda'(v)=
  (\lambda/\gamma)^{|v|}-(\lambda/\gamma)^{|\parent{v}|}=(\lambda/\gamma)^{|v|-1}
  (\lambda/\gamma -1).
\]
Hence, $f_\lambda \in \Lip$ if and only if $|\lambda|\leq \gamma$ and
$f_\lambda \in \Lip_0$ if and only if $|\lambda| < \gamma$ or
$\lambda=\gamma$. The result now follows immediately.
\end{proof}

With this, we are ready to prove the following theorem.

\begin{theorem}\label{th:sigmaB_L}
Let $T$ be a rooted homogeneous tree of order $\gamma$. If $B$ is the
backward shift on $\Lip$, then
\[
\sigma(B)=\sigma_{\operatorname{ap}}(B)=\sigma_{\operatorname{p}}(B)=
\{ \lambda \in \C \, : \, |\lambda| \leq \gamma \}.
\]
\end{theorem}
\begin{proof}
First, we will compute the spectral radius of $B$. By Theorem
\ref{th:Bn} if $\gamma=1$ then
\[
  r(B)=\lim_{n \to \infty} \| B^n \|^{1/n}=\lim_{n \to \infty}
  (n+1)^{1/n}=1=\gamma.
\]
If $\gamma \geq 2$, Theorem \ref{th:Bn} gives
\[
  r(B)=\lim_{n \to \infty} \| B^n \|^{1/n}=\lim_{n \to \infty}
  ((2n+1)\gamma^n -(2n)\gamma^{n-1})^{1/n}= \lim_{n \to \infty}
  \gamma^{(n-1)/n} ((2n+1)\gamma-2n)^{1/n} =  \gamma.
\]
Therefore, $\sigma(B)\subseteq \{ \lambda \in \C \, : \, |\lambda|
\leq \gamma \}$. This and the previous theorem imply that
\[
  \{ \lambda \in \C \, : \, |\lambda| \leq \gamma \} \subseteq
  \sigma_{\operatorname{p}}(B) \subseteq \sigma_{\operatorname{ap}}(B) \subseteq \sigma(B) \subseteq
  \{ \lambda \in \C \, : \, |\lambda| \leq \gamma \},
\]
and hence the result follows.
\end{proof}

We obtain a similar result for the backward shift on $\Lip_0$.

\begin{theorem}\label{th:sigmaB_L0}
Let $T$ be a rooted homogeneous tree of order $\gamma$. If $B$ is the
backward shift on $\Lip_0$, then
\[
\sigma(B)=\sigma_{\operatorname{ap}}(B)= \{ \lambda \in \C \, : \, |\lambda| \leq \gamma \}.
\]
\end{theorem}
\begin{proof}
As was the case in the theorem above, by Theorem \ref{th:Bn0} we have
$r(B)=\gamma$ and hence
\[
  \sigma(B)\subseteq \{ \lambda \in \C \, : \, |\lambda| \leq \gamma
  \}.
\]
This, and Theorem \ref{th:sigmapB} give that
\[
  \sigma(B) =  \{ \lambda \in \C \, : \, |\lambda| \leq \gamma
  \}.
\]
Since, for any operator $A$ we have $\partial \sigma(A) \subseteq \sigma_{ap}(A)$
(see, e.g. \cite[p.~210]{Conway}), we have that
\[
  \{ \lambda \in \C \, : \, |\lambda| = \gamma \} \subseteq \sigma_{ap}(B),
\]
and, again, by Theorem \ref{th:sigmapB}, we obtain
\[
  \sigma_{ap}(B) =  \{ \lambda \in \C \, : \, |\lambda| \leq \gamma \},
\]
which completes the proof.
\end{proof}

With the previous result showing what the spectrum of the backward
shift is, we can determine the point spectrum.

\begin{theorem}\label{th:sigmap_L0}
Let $T$ be a rooted homogeneous tree of order $\gamma$. If $B$ is the backward shift on $\Lip_0$, then
\[
\sigma_{\operatorname{p}}(B)= \{ \lambda \in \C \, : \, |\lambda|< \gamma \} \cup \{ \gamma \}.
\]
\end{theorem}
\begin{proof}
By Theorems \ref{th:sigmapB} and \ref{th:sigmaB_L0}, it suffices to
show that if $\lambda \neq \gamma$ and $|\lambda|=\gamma$, then
$\lambda \notin \sigma_p(B)$. So let $\lambda \neq \gamma$ with
$|\lambda|=\gamma$  and assume then that $Bf=\lambda f$ for a nonzero
$f \in \Lip_0$.

First, since $f$ is not zero, there exists a vertex $w^*$ such that $f(w^*)\neq
0$. Dividing by a constant, if necessary, we may assume that $f(w^*)=1$

Now, we claim that for all $n \in \N$ there exists $v \in
\nChi{w^*}{n}$ with $|f(v)|\geq 1$. Indeed, suppose this was not the
case. Then, for some $m \in \N$ we would have $|f(v)|< 1$ for all $v
\in \nChi{w^*}{m}$. But since $B^m f = \lambda^m f$, we have
\[
\lambda^m f(w^*)=\sum_{v \in \nChi{w^*}{m}} f(v)
\]
and hence we obtain
\[
|\lambda|^m = |\lambda^m f(w^*)| \leq \sum_{v \in \nChi{w^*}{m}}
|f(v)| < \gamma^m 1 = |\lambda|^m,
\]
which is a contradiction, so the claim is true.

Now, since $f \in \Lip_0$, there exists $N \in \N$ such that, for all
$|v|\geq N$ we have
\[
  |f'(v)|=|f(v)-f(\parent{v})| < \frac{|\gamma - \lambda|}{2\gamma}
\]
By the claim, there exists $u^* \in \nChi{w^*}{N}$ with $|f(u^*)|\geq
1$. Hence
\[
  \sum_{u \in \Chi{u^*}} (f(u) - f(\parent{u})) = (Bf)(u^*) - \gamma
  f(u^*) = (\lambda- \gamma) f(u^*). 
\]
But then,
\[
  |\lambda- \gamma| \,  |f(u^*)| \leq   \sum_{u \in \Chi{u^*}} |f(u) -
  f(\parent{u})| < \gamma \frac{|\gamma - \lambda|}{2\gamma},
\]
since every $u \in \Chi{u^*}$ satisfies $|u|>N$. But the last display
implies that $|f(u^*)| < \frac12$, which is a contradiction. Hence
there cannot be $\lambda \neq \gamma$ with $|\lambda|=\gamma$ and
$Bf=\lambda f$ for a nonzero $f \in \Lip_0$, which completes the proof
of the theorem.
\end{proof}

\section{Hypercyclicity}\label{sec_hyper}

In \cite{CoMA1}, it is shown that $\Lip$ (with an equivalent norm) is
not separable, while $\Lip_0$ is separable (this was originally shown
in \cite{CoEa1}). So, in order to study hypercyclicity of operators,
we need to restrict ourselves to $\Lip_0$, which we do from now on.

First, we get rid of the question of whether $S$ is hypercyclic. It is
not since the norm of $S$ is one and therefore $S$ can never be
hypercyclic. We offer an alternative proof.

\begin{theorem}
Let $T$ be a rooted, countably infinite and locally finite tree and let $S$ be the forward
shift on $\Lip_0$. Then $S$ is not hypercyclic.
\end{theorem}
\begin{proof}
If $S$ were hypercyclic, then there would exist $f \in
\Lip_0$ and a natural number $N$ such that
\[
\| S^{N} f - \chi_{\{\root\}} \| < \frac{1}{2},
\]
where $\chi_{\{\root\}}$ is the characteristic function of the root
$\root$. The definition of the norm in $\Lip$ then would imply that
\[
  | (S^{N} f)(\root) - 1 | \leq  \| S^{N} f - \chi_{\{\root\}} \| < \frac{1}{2}.
\]
But, since $(S^n f)(\root)=0$ for every $n \in \N$, this is a
contradiction. Therefore $S$ is not hypercyclic. 
\end{proof}

We will use the following lemma, which is proved in \cite{CoMA2}.

\begin{lemma}
Let $X$ be the set of all functions in $\Lip_0$ with finite support.
Then $X$ is dense in $\Lip_0$.
\end{lemma}

The following definition will be useful to characterize hypercyclicity.

\begin{defi}
Let $T$ a rooted tree and $v \in T$. We say that $S_v$ is a {\em free end (at
  $v$)} if for all $w \in S_v$ we have $\gamma(w)=1$.
\end{defi}

\begin{figure}[!htbp]
\centering
\begin{tikzpicture}[level 1/.style={sibling distance=2.5cm}, level 2/.style={sibling distance=1.5cm}, scale=0.7,emph/.style={edge from parent/.style={black, dashed,draw}}]
\node[draw,circle] {$\root$} [grow'=down]
child {node[draw,circle]  {}
    child{node[draw,circle] {}
      child{node[draw,circle] {}
      child[emph]{}}
        child{node[draw,circle] {}
          child{node[draw,circle] {}
          child[emph]{}}
        child{node[draw,circle] {}
           child[emph]{}}
            }
         }  
         child{node[draw,circle]{}
         child[emph]{} }
  } 
child {node[draw,circle]  {}
  child {node[draw,circle] {}
  child[emph]{}}
    child {node[arn_r, draw,circle] {$v$}
        child{node[arn_r, draw,circle] {}
            child{node[arn_r, draw,circle] {}
                child{node[arn_r, draw,circle] {}
                    child[emph]{}
                }
            }
        }
     } 
};
\end{tikzpicture}
\caption{Free end starting at the vertex $v$.}
\end{figure}

Recall that $T^n$ denotes the set of vertices that have $n$-parents;
i.e., $v \in T^n$ if there exists $u \in T$ with $v \in
\nChi{u}{n}$. Also, recall that $\gamma(u,n)$ denotes the number of
vertices in the set $\nChi{u}{n}$.

We define the function  $\beta: T \times \N \to \R$ as
\[
  \beta(v,n)=
  \begin{cases}
    \frac{1}{\gamma(\nparent{v}{n},n)}, & \text{ if } v \in T^n, \\
    0,  & \text{ if } v \notin T^n.
    \end{cases}
\]
The following lemma will be used later.
\begin{lemma}\label{le:beta}
  Let $T$ be a rooted, countably infinite and locally finite tree. If $T$ is homogeneous
  by sectors and has no free ends then
  \[
    \sup_{w \in T} \beta(w,n) \to 0 \quad \text{ as } n \to \infty.
  \]
 \end{lemma}
 \begin{proof}
Since $T$ is homogeneous by sectors, there exists $M \in \N$ such that
for every $v \in T$ with $|v|=M$, we have $\gamma(v)=\gamma(u)$
for every $u \in S_v$. Since $T$ is locally finite, there exist
finitely many such $v$, say $v_1, v_2, \dots, v_r$. For each $j=1, 2,
\dots, r$, define $\mu_j:=\gamma(v_j)$. Since $T$ has no free ends, $\mu:=\min\{ \mu_1,
\mu_2, \dots, \mu_r \} \geq 2$.

Let $w\in T$ and let $n \geq 2 M$.

\begin{itemize}
\item If $|w|<n$, then $w \notin T^n$ and hence
  \[
    \beta(w,n)=0.
  \]
  
\item If $|w|\geq M+n$, then $|\nparent{w}{n}| \geq M$ and hence
  $\gamma(\nparent{w}{n},{n}) \geq \mu^n$. Therefore,
  \[
    \beta(w,n) \leq \frac{1}{\mu^n}.
  \]
  
\item If $n \leq |w| < M+n$, let $k=|w|$. Then, since $k-M < n$, we
  have $\nChi{\nparent{w}{k-M}}{k-M} \subseteq
  \nChi{\nparent{w}{n}}{n}$ and hence
  \[
    \gamma(\nparent{w}{k-M},k-M) \leq \gamma(\nparent{w}{n},n).
  \]
  But clearly $\mod{\nparent{w}{k-M}}=M$ which implies that
  $\gamma(\nparent{w}{k-M},k-M) \geq \mu^{k-M}$ and hence $\beta(w,n)
  \leq \frac{1}{\mu^{k-M}}$. Since $k >n$ this gives
  \[
    \beta(w,n) \leq \frac{1}{\mu^{n-M}}
  \]  
\end{itemize}
Hence, for every $w \in T$, we have
  \[
    \beta(w,n) \leq \frac{1}{\mu^{n-M}}
  \]
  if $n > 2M$. Therefore, if $n > 2M$, then
  \[
    \sup_{w \in T} \beta(w,n) \leq \frac{1}{\mu^{n-M}}
  \]
and thus  
  \[
    \sup_{v \in T} \beta(v,n) \to 0 \quad \text{ as } n \to \infty. \qedhere
  \]
\end{proof}

We can now give a sufficient condition for hypercyclicity of $B$.

\begin{theorem}
Let $T$ be an countably infinite and locally finite tree and assume that $B$ is
bounded on $\Lip_0$. If $T$ has no free ends, then $B$ is hypercyclic.  
\end{theorem}
\begin{proof}
To show hypercyclicity of $B$, we will use the Hypercyclicity
Criterion (Theorem \ref{th:hyp_cri}). Let $X$ be the set of all
functions with finite support and for each $n \in \N$, define the
function $R_n: X \to \Lip$ as
\[
  (R_n f)(v) =\begin{cases}
    \beta(v,n) f(\nparent{v}{n}), & \text{ if } v \in T^n, \text{ and } \\
    0, & \text{ if } v \notin T^n.
  \end{cases}
\]
(Observe that $R_n f$ also has finite support.)

\begin{enumerate}
\item First, let $f \in X$. Choose $N \in \N$ such that $f(v)=0$ for
  all $v$ with $|v|\geq N$. Then, for all $v \in T$,
  \[
    (B^nf)(v)=\sum_{w \in \nChi{v}{n}} f(w) =0
  \]
  if $n> N$ and, similarly, for all $v \in T^*$ we have
  $(B^nf)(\parent{v})=0$ if $n>N$. Hence $(B^nf)'=0$ as long as $n >
  N$. Therefore $B^n f \to 0$, as $n \to \infty$, as desired.

\item Let $f \in X$ and $v \in T$. Since $f$ is of finite support,
  there exists $M >0$ such that $|f(v)|\leq M$ for all $v \in T$.
\begin{itemize}
  \item If $|v|<n$, then $(R_n f)(v)=0$; while if $0<|v|<n$, then $(R_n
  f)(\parent{v})=0$. Hence $(R_n f)'(v)=0$ if $|v|<n$.
\item If $|v|=n$, then, since $\parent{v} \notin T^n$, then
  \begin{align*}
   |(R_n f)'(v)|
    & = | (R_n f)(v) - (R_n f)(\parent{v}) | \\
    & = \beta(v,n) |f(\nparent{v}{n})|  \\
    & \leq M \beta(v,n).
  \end{align*}
  
  \item If $|v| > n$. Then,
    \begin{align*}
      |(R_n f)'(v)|
      & = | (R_n f)(v) - (R_n f)(\parent{v})| \\
    & \leq \beta(v,n) |f(\nparent{v}{n})| + \beta(\parent{v},n)
    |f(\nparent{v}{n+1})| \\
    & \leq M (\beta(v,n) + \beta(\parent{v},n)).
  \end{align*}
\end{itemize}
  Therefore, for all $v \in T$, we have
  \[
      |(R_n f)'(v)| \leq 2M \sup_{v \in T} \beta(v,n)
  \]
  and hence
  \[
    \| R_n f \| \leq 2M \sup_{v \in T} \beta(v,n).
  \]
    
Therefore, since $T$ has no free ends, by Lemma \ref{le:beta} (since
$B$ is bounded and hence $T$ is homogeneous by sectors) we have that
$ R_n f  \to 0$, as desired.

\item Now, let $v \in T$. We then have
  \begin{equation*}
    \begin{split}
      (B^n (R_n f))(v)
      & = \sum_{w \in \nChi{v}{n}} (R_n f)(w) \\
      & = \sum_{w \in \nChi{v}{n}} \frac{1}{\gamma(\nparent{w}{n},n)}  f(\nparent{w}{n}) \\
      & = \sum_{w \in \nChi{v}{n}} \frac{1}{\gamma(v,n)}  f(v) \\
      & = f(v).
    \end{split}
  \end{equation*}
  Therefore, $B^n R_n f \to f$ as $n \to \infty$, as desired.
\end{enumerate}
Since all conditions in the hyperciclicity criterion hold, it follows
that $B$ is hypercyclic.
\end{proof}

The following theorem shows that the condition on the previous theorem
actually characterizes hypercyclicity.

\begin{theorem}
Let $T$ be an countably infinite and locally finite tree and assume that $B$ is
bounded on $\Lip_0$. If $B$ is hypercyclic, then $T$ has no free ends.
\end{theorem}
\begin{proof}
  Asumme that $T$ has a free end. Let $v^*$ be a vertex on the free
  end such that $\gamma(v^*)=1$ and $\gamma(\parent{v^*})=1$.  Then,
  for each $n \in \N$ each of the sets $\nChi{v^*}{n}$ and $\nChi{\parent{v^*}}{n}$ has a
  unique element. 

Since $B$ is hypercyclic there exists a hypercyclic vector $f$. In
fact, by the density of the hypercyclic vectors, we may assume that
$\| f \| < \frac{1}{2}$. Let $\chi_{\{v^*\}} \in\Lip_0$ be the
characteristic funtion of $v^*$. By hypercyclicity of $B$, there
exists $N \in \N$,  such that
\[
  \| B^N f - \chi_{\{v^*\}} \| < \frac{1}{2}.
\]
But then
\begin{align*}
  \| B^N f - \chi_{\{v^*\}} \|
  &\geq  \mod{ (B^N f)'(v^*) - \chi_{\{v^*\}}'(v^*) } \\
  &= \mod{ \sum_{w \in \nChi{v^*}{N}} f(w) -  \sum_{w \in
    \nChi{\parent{v^*}}{N}} f(w) - 1 } \\
  &= \mod{ f(w_N) - f(\parent{w_N}) - 1 },
\end{align*}

where $w_N$ is the unique element in the set $\nChi{v^*}{N}$. Hence,
\[
  \mod{ f(w_N) -f(\parent{w_N}) - 1 } < \frac{1}{2}
\]
and therefore
\[
  \frac{1}{2} <  \mod{ f(w_N) -f(\parent{w_N}) }.
\]
But, since $\| f \| < \frac{1}{2}$, we have
\[
  \mod{ f(w_N) -f(\parent{w_N}) } < \frac{1}{2},
\]
which is a contradiction. Therefore, $T$ cannot have free ends.
\end{proof}

We summarize the previous two theorems to obtain a full
characterization of the hypercyclicity of $B$.

\begin{theorem}\label{th:main_hyper}
Let $T$ be a rooted, countably infinite and locally finite tree and assume that $B$ is
bounded on $\Lip_0$. Then $B$ is hypercyclic if and only if $T$ has no free ends.
\end{theorem}


\end{document}